\documentclass{amsart}
\usepackage[all,cmtip]{xy}
\usepackage{latexsym}                                                          
\usepackage{amssymb}
\usepackage{epsfig}
\usepackage{psfrag}
\usepackage{amsthm}
\usepackage{amscd}
\usepackage{amsmath}
\usepackage{amsfonts}
\usepackage{graphics}
\usepackage[all]{xy}
\usepackage{comment}
\title[Complexified Morse functions]       
{The Picard-Lefschetz theory of complexified Morse functions}

\author{Joe Johns}
\email{jjohns@cims.nyu.edu}

\newtheorem*{namedtheorem}{\theoremname}
\newcommand{\theoremname} {testing}
\newenvironment{named}[1]{\renewcommand{\theoremname}{#1}\begin{namedtheorem}}
{\end{namedtheorem}}

\newtheorem{thm}{Theorem}[section]

\newtheorem{problem}[thm]{Problem}
\newtheorem{prop}[thm]{Proposition}
\newtheorem{lemma}[thm]{Lemma}

\theoremstyle{definition}

\newtheorem{remark}[thm]{Remark}

\newcommand{\bbZ}{\mathbb{Z}}

\newcommand{\bbR}{\mathbb{R}}
\newcommand{\bbC}{\mathbb{C}}

\def\co {\colon\thinspace}
\newcommand{\into}{\longrightarrow}
\renewcommand{\dim}{\textup{dim \,}}

\begin{document}

\begin{abstract}    
Given a closed manifold $N$ and a self-indexing Morse function $f\co  N \into \bbR$ with up to four distinct Morse indices, we construct a symplectic Lefschetz fibration 
$\pi\co E \into \bbC$ which models the complexification of $f$ on the disk cotangent bundle, 
$f_\bbC \co  D(T^*N) \into \bbC$, when $f$ is real analytic. 
By construction, $\pi\co  E \into \bbC$ comes with an explicit regular fiber $M$ and  vanishing spheres  $V_1, \ldots, V_m \subset M$, one for each critical point of $f$. Our main result is that $\pi\co  E \into D^2$ is a good model for $f_\bbC\co D(T^*N) \into \bbC $ in the sense that $N$ embeds in $E$ as an exact Lagrangian submanifold, and in addition, $\pi|_{N} = f$ and 
$E$ is homotopy equivalent to $N$.  There are several potential applications in symplectic topology, which we discuss in the introduction.
\end{abstract}

\maketitle



\section{Introduction}\label{intro} Complexifications in various forms have been considered for some time in many different fields, as in for example \cite{BW}, \cite{G}, \cite{K}, \cite{Lempert}, \cite{A'C}, \cite{AMP}. The central problem is usually to understand the relation between the complexified object and its real counter part. In this paper we are concerned with the relation between Morse functions and their complex versions, Lefschetz fibrations. Historically, Morse functions have of course played a central role in elucidating the differential topology of manifolds.
On the other hand, Lefschetz fibrations played a large role in the study of the topology of algebraic varieties. They have a formally similar definition to Morse functions (they are proper holomorphic maps on complex analytic varieties with generic singularities modeled on $z_1^2 + ... + z_n^2$), but they are in fact quite different in flavour.
The most obvious difference one can point to is that Morse functions have regular level sets which differ in topological type, whereas all the regular fibers of a Lefschetz fibrations are diffeomorphic.
(See, for example, \cite{Bott}, \cite{L} for surveys of Morse theory and the theory of Lefschetz fibrations.)
\\
\newline  Here is one way to complexify a Morse function to get a Lefschetz fibration. 
(Actually, the method we now present is somewhat naive, see remark \ref{pivsfC}, but it serves as a useful starting point.) Suppose that $N$ is a real analytic manifold and $f\co N \into \bbR$ is a real analytic Morse function.
Then, in local charts on $N$, $f$ is represented by some convergent power series with real coefficients;
if we complexify these local power series to get complex analytic power series (with the same coefficients) we obtain a complex analytic map on the disk bundle of $T^*N$ of some small radius $\epsilon>0$,
$$f_\bbC\co D_\epsilon(T^*N) \into \bbC,$$
called the \emph{complexification} of $f$. In favorable circumstances, $f_{\bbC}$ can be regarded as a Lefschetz fibration. (The main issue which may prevent $f_\bbC$ from being a  Lefschetz fibration is that parts of some fibers may be ``missing''; this will happen if  the parallel transport vector-field between regular fibers is \emph{incomplete}. See remark \ref{pivsfC} for more on this.) 
Actually, we want to think of $f_\bbC$ as a symplectic Lefschetz fibration,
which is an extension of the notion  to the symplectic category due to Donaldson
(see \cite{DA,DB}). There are some slightly different possible definitions and we follow the one in 
\cite{LES} (which also satisfies the definition in \cite{S08}); this is designed for work in the category of exact symplectic manifolds (with nonempty boundary). 
Roughly, a symplectic Lefschetz fibration $\pi\co E \into \bbC$ is a fiber bundle with symplectic fibers, but with some isolated singular points modeled on $z_1^2 + ... + z_n^2$.  
\\
\newline It is natural to ask: 
\emph{To what extent can the complex topology of $f_\bbC$ be described in terms of  the Morse theory of $f$?}
The central problem here is to analyze the Picard-Lefschetz theory of $f_\bbC$, and we are particularly interested the symplectic view point (as opposed to classical topology, which deals with cycles):

\begin{problem}  \label{complexify}
Describe the generic fiber of $f_\bbC$ as a symplectic manifold $M$, and describe its vanishing cycles as Lagrangian spheres in $M$. 
\end{problem}

Basically this entails reconstructing the picture of the regular complex fiber $M$ starting only with knowledge of the Morse theory on the real part $N \subset T^*N$.  This problem is very natural from the point of view of singularity theory, and indeed our approach  is greatly influenced by work of A'Campo \cite{A'C} which treats the case where $f$ is a real polynomial on $N = \bbR^2$, 
and $f_\bbC\co \bbC^2 \into \bbC$ is the same polynomial on $\bbC^2$. 
\\
\newline In our approach  we first take  a Riemannian metric $g$ such that $(f,g)$ is Morse-Smale and  we only assume $(N,f,g)$ is smooth, not necessarily real analytic. Given the corresponding handle decomposition of $N$, we construct an exact symplectic manifold $M$ of dimension $2 \dim N-2$ together with some exact Lagrangian spheres $L_1, \ldots, L_m \subset M,$ one for each critical point of $f$.  Then, given $(M, L_1, \ldots, L_m)$, there is a unique (up to deformation) symplectic Lefschetz fibration $\pi\co E \into \bbC$, with generic fiber $M= \pi^{-1}(b)$ and vanishing spheres  $L_1, \ldots, L_m$ (see \cite[\S 16e]{S08}).
Theorem $A$ below shows that $(E,\pi)$ and $(D(T^*N), f_\bbC)$ have the same key features. In this sense $(E,\pi)$ is a good model for $(D(T^*N), f_\bbC)$ and $(M, L_1, \ldots, L_m)$ is very likely a correct answer to problem \ref{complexify}. In any case, for applications (see \S \ref{motivations})
we need not use $f_\bbC$; instead we will always use $(E,\pi)$. Thus the question of whether $E \cong D(T^*N)$ and $\pi \cong  f_\bbC$ is not pressing for the moment; still this question remains an interesting one and we intend to pursue it in future work, see remark \ref{pivsfC} below for more on that.

\begin{named}{Theorem A} Assume $N$ is a smooth closed manifold and  $f\co N \into \bbR$ is self-indexing Morse function with either two, three, or four critical values: $\{0,n\}$, $\{0,n,2n\}$, or $\{0,n,n+1,2n+1\}$. Let $(M, L_1 \ldots, L_m)$ be the data from the construction we discussed above (which depends in addition on a Riemannian metric $g$ on $N$ such that $(f,g)$ is Morse-Smale).  Let $\pi\co E \into \bbC$ be the corresponding symplectic Lefschetz fibration 
with fiber $M$ and vanishing spheres $(L_1, \ldots, L_m)$.
Then, $N$ embeds in $E$ as an exact Lagrangian submanifold so that: 
\begin{itemize}
\item all the critical points of $\pi$ lie on $N$, and in fact $Crit(\pi) = Crit(f)$; 
\item $\pi(N)$ is a closed subinterval of $\bbR$; and,
\item $\pi|N = f\co N \into \bbR$ (up to reparameterizing $N$ and $\bbR$ by diffeomorphisms).
\end{itemize}
\end{named}

Theorem $A$ appears below as Theorem \ref{main}. As a step towards proving $E \cong D(T^*N)$, we also give a fairly detailed sketch that $E$ is homotopy equivalent to $N$ in Proposition \ref{Prophomotopyequiv}. (See also remark \ref{pivsfC} for a different sketch of the stronger statement, $E \cong D(T^*N)$, which is less detailed.)

\begin{remark} The contruction of $(M, L_1, \ldots, L_m)$ and $(E,\pi)$, and the proof of Theorem $A$, all work equally well for non-closed manifolds $N$.
\end{remark}

The main point of  Theorem $A$ is the construction of  the exact Lagrangian embedding $N \subset E$. The other conditions on $N$ are rather natural, and in fact their main use here is to guide the construction of $N \subset E$.
The hypotheses on $f$ ensure that the contruction of $M$ is relatively easy. 
In a sequel paper, to appear later, we will give a detailed treatment of the  construction of $M$ in general, which is more complicated. (See \S \ref{Mgeneral} of this paper for a partial sketch of that.) The proof of Theorem $A$ will then carry over easily to the general case. (Indeed, even in this paper our proof of Theorem $A$ focuses first on the case where $f$ has just three Morse indices  
(see \S \ref{basics}-\ref{ConstructionN}). Then we explain how the proof carries over easily to the case where $f$ has four Morse indices in \S \ref{3-manifolds}.)
Later, we will also flesh out the ideas relating $(D(T^*N), f_\bbC)$ and $(E,\pi)$ sketched in remark \ref{pivsfC} below. For more about the contents and organization of this paper, see \S \ref{overview} below.

\begin{remark}\label{pivsfC} The precise relationship of $\pi$ to $f_\bbC$ is not so obvious, first, because $\pi$ apparently depends on the metric $g$ whereas $f_\bbC$ does not, and second, because $\pi$ can arise from non-analytic data $(f,g)$, whereas $f_\bbC$ must have an analytic datum $f$. Finally, and most importantly, $f_\bbC$ is not really a  Lefschetz fibration  in the strict sense if one constructs it in the naive way we described. The reason is the regular fibers of $f_\bbC$ are not 
going to be diffeomorphic. Indeed, for a given real regular value $x \in \bbR$, the complex level set
$f_\bbC^{-1}(x)$ will be symplectomorphic to  $D(T^*(f^{-1}(x)))$, which is the disk cotangent bundle of the corresponding real level set $f^{-1}(x)$. But this fiber is much too small: If $f_\bbC$ was
a Lefschetz fibration with \emph{complete} parallel transport vector fields (where the connection is given by the symplectic orthogonal to the fibers), then all the regular fibers would be symplectomorphic, hence each regular fiber should contain \emph{all} the real regular level sets of $f$ as Lagrangian submanifolds (because we could parallel transport all these Lagrangian submanifolds into any fixed regular fiber). One can think of $\pi\co E \into \bbC$ as a larger (genuine) Lefschetz fibration which does have this property.
We conjecture that the correct picture relating $(E, \pi)$ and $(D(T^*N), f_\bbC)$ 
is that $D(T^*N)$ embeds into $E$ as a kind of diagonal subset intersecting each regular fiber $\pi^{-1}(x)$, $x \in \bbR$, in a Weinstein neighborhood of the real level set $f^{-1}(x)$, as we said before. Then $E$ should be conformally exact symplectomorphic to $D(T^*N)$ by a retraction along a Liouville type vector field (namely, a variably re-scaled version of the symplectic lift of $\frac{\partial}{\partial x}$ from the real line), and $\pi|_{D(T^*N)}$ should be deformation equivalent to $f_\bbC$. 
\end{remark}
\subsection{Overview}\label{overview}  The main purpose of this paper is to prove Theorem $A$, see Theorem \ref{main} below. There are three  steps in the proof. First, in \S \ref{fiber}, we construct the regular fiber $M$ and vanishing spheres $L_1, \ldots, L_m$. (The construction is based on some techniques explained in \S \ref{MorseBott}.)
 Second, in \S \ref{basics}-\ref{computingVC}, we construct a symplectic Lefschetz fibration $\pi\co 
E \into D^2$ with this regular fiber and these vanishing spheres. (See \S \ref{Esketch} to see why this
construction is not immediate.) Third, in \S \ref{ConstructionN}, we construct an exact Lagrangian embedding $N \subset E$ satisfying the conditions in Theorem $A$. (See \S \ref{Nsketch} for a sketch.)
\\
\newline 
Once the construction of $N \subset E$ is carried out in detail in one particular case, it is  straight-forward to see how it works in any other case. For this reason we first concentrate on the case where $f$ has three distinct Morse indices $0,n,2n$, focusing on the case where $\dim N = 4$; 
this takes up the bulk of the paper, \S \ref{basics}-\ref{ConstructionN}. 
In \S \ref{3-manifolds} we explain how  the proof carries over easily to the case when $f$ has four distinct Morse indices $0,n,n+1, 2n+1$, focusing on the case $\dim N =3$.
In general, for any fixed number of Morse indices for $f$, all three steps in the proof  are  essentially the same as the dimension of $N$ varies. 
(See \S \ref{2nfiber}, for example, to see why this is so.) That is why we look at the slightly
more concrete cases where $\dim N =4$ and $\dim N =3$. (We focus on the case $\dim N = 4$ rather than $\dim N = 2$ because the latter case is actually a little confusing from the general point of view, because of its very low dimension.)

\begin{remark}
 As it happens, closed 4-manifolds admitting handle decompositions without 1- and 3-handles form a fairly rich class of examples. This class includes $\bbC P^2$, $\overline{\bbC P^2}$,  and elliptic surfaces (assuming simply-connected with sections, see \cite{GS}), and it is closed under connected sum (so blow up in particular). 
It is not known if it includes all simply-connected closed 4-manifolds, or even all hypersurfaces in $\bbC P^3$; this has been a long standing question. (Recently, however, a famously conjectured  counter-example, the Dolgachef surface \cite{HKK}, has  been proved by Akbulut \cite{Ak} to be in this class as well.) See \cite{HKK}, \cite{Kirby}, \cite{GS}  for more extensive discussions.
\end{remark}

\subsubsection{Sketch of the construction of $(E, \pi)$} \label{Esketch}
Given $(M, L_1, \ldots, L_m)$ there is a standard (and fairly easy) construction for producing a Lefschetz fibration with fiber $M$ and vanishing spheres $L_1, \ldots, L_m$. (See \cite[\S 16e]{S08}.)
We, however, construct $(E, \pi)$ in a different, more explicit, way in order to facilitate the construction of $N \subset E$. Actually, most of the difficulty for us  is concentrated in this non-standard construction of  $(E,\pi)$ (see \S \ref{basics}-\ref{computingVC}) after which the contruction of $N \subset E$
is relatively easy (see \S \ref{ConstructionN}). 
\\
\newline Our strategy for constructing $(E, \pi)$ is as follows. For simplicity of notation assume that $f$ has just one critical point of each index, say $x_0, x_2, x_4$. 
Our aim is for  $\pi\co E \into D^2$  to have exactly three critical values $c_0 < c_2 < c_4$; these will lie on the real line and correspond to the critical values of $f$. We first construct three Lefschetz fibrations, one for each critical value $c_i$: 
$$\pi_i\co E_i \into D_s(c_i), \, i=0,2,4,$$
where $D_s(c_i)$ is a small disk around $c_i$ of radius $s>0$.
The aim is to construct $(E,\pi)$ so that $(E_i, \pi_i)$ is equal to the restriction of $\pi$ to a small disk around $c_i$, namely $ (E|_{D_s(c_i)}, \pi|_{D_s(c_i)}) = (E_i, \pi_i)$.
Each $\pi_i$ will have a unique critical point, say $x_i$, with just one vanishing sphere corresponding to $x_i$. In this case there is a standard construction (namely \cite[lemma 1.10]{LES}) which we use to produce each $\pi_i$ (see \S \ref{pis}). 
The regular fiber of each $\pi_i$, say $M_i$, will be exact symplectomorphic to $M$. 
The key new ingredient in our construction of $(E, \pi)$ is a twisting operation 
$$M \leadsto T_{\pm \pi/2}(M)$$
that we use to define $M_2 = T_{\pi/2}(M)$. Roughly, $T_{\pi/2}(M)$ is defined by deleting a 
Weinstein neighborhood of some chosen Lagrangian sphere and then gluing it back in with 
a $\pi/2$ twist. (See \S \ref{pi2} for the precise definition; we will explain its main properties in a moment.) By construction, $\pi_0$,$\pi_2$,$\pi_4$  have regular fibers given by canonical isomorphisms as follows:
 $$\pi_0^{-1}(c_0+s) \cong M,$$
 $$\pi_2^{-1}(c_2+s) \cong M_2,$$
 $$\pi_4^{-1}(c_4-s) \cong M_2.$$ 
The main point of the twisting operation  $M \leadsto T_{\pi/2}(M)$ is that 
there is a canonical isomorphism $\pi_2^{-1}(c_2-s)\cong M$ (see lemma \ref{leftfiber=M})
and a natural exact symplectomorphism 
$$\tau\co M \into M_2 = T_{\pi/2}(M)$$
such that, under the canonical isomorphisms $\pi_2^{-1}(c_2-s)\cong M$ and 
$\pi_2^{-1}(c_2+s)\cong M_2$, $\tau$  coresponds precisely to 
the symplectic  transport map
$$\pi_2^{-1}(c_2-s) \into  \pi_2^{-1}(c_2+s)$$ 
along a half circle in the lower half plane from $c_2-s$ to $c_2+s$ (see lemma \ref{halftwist}).
Thus, $\tau$ is somewhat like a ``half generalized Dehn-twist'' (compare with \cite[\S 16c]{S08}).
\\
\newline To construct  $\pi\co  E \into D^2$ (see \S \ref{pi}), first consider the boundary connected sum of $D_s(c_2)$ to $D_s(c_0)$, where we delete neighborhoods of $c_2-s$ and $c_0+s$ and then identify
along the resulting boundaries. Then, similarly, we boundary connect-sum $D_s(c_2)$ to $D_s(c_4)$ at $c_2+s$ and $c_4-s$. This yields 
$$S = D_s(c_0) \#D_s(c_2) \# D_s(c_4) \cong D^2.$$
Then $(E,\pi)$ is defined to be the Lefschetz fibration over $S \cong D^2$
obtained by identifying the corresponding fibers of $(E_2,\pi_2)$ at $c_2-s$ and $c_2+s$
respectively with the fiber of $(E_0, \pi_0)$ at $c_0+s$,  and the fiber of $(E_4, \pi_4)$ at $c_4-s$.
For this we use the identifications  we mentioned above
$$\pi_0^{-1}(c_0+1) \cong M \cong \pi_2(c_2-1) $$
$$\pi_2^{-1}(c_2+1) \cong M_2 \cong \pi_4(c_4-1).$$
(This operation on Lefschetz fibrations is called  \emph{fiber connected sum}--see \S \ref{fiberconnectsum}.)

\subsubsection{Sketch of the construction of $N \subset E$}\label{Nsketch}

To construct the required exact Lagrangian embedding $N \subset E$ we use a handle-type decomposition of $N$ induced by the Morse function $f$, which is due to Milnor \cite[pages 27-32]{Milnor}.
For example, let us take the usual handle-decomposition of $N = \bbR P^2$ with three handles. Then,  in the corresponding Milnor decomposition there are four pieces:
\begin{figure}
\begin{center}
 \includegraphics[width=4in]{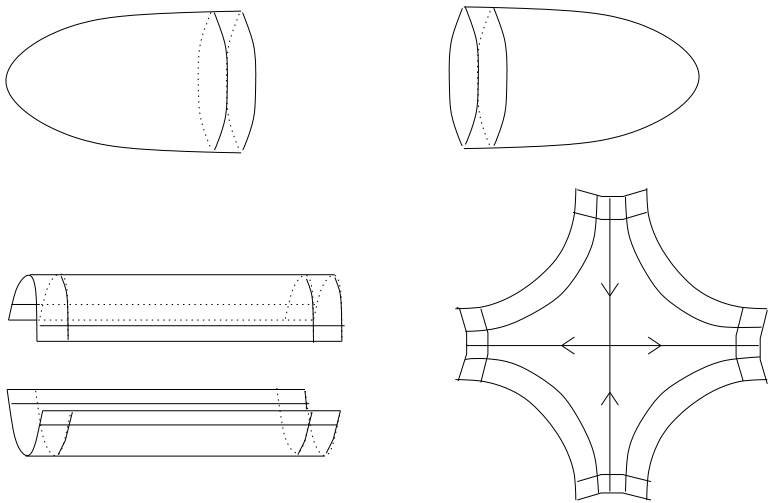} 
\caption{In the case $N= \bbR P^2$, the pieces $N_0$,$N_2$ (top), $N_1^{triv}$ (bottom left), $N_1^{loc}$ (bottom right).
The overlap regions are also indicated.}
\label{handledec}
\end{center}
\end{figure} 
First, there are $N_0 = D^2$ and $N_2 = D^2$, which are the same as the usual 0- and 2-handles. Then there is 
$$N_1^{loc} = \{ x \in \bbR^2 :    |q_2(x)| \leq 1,\, |x|^4  - q_2(x)^2 \leq \delta \},$$
where $\delta>0$ is some small number and $q_2(x) = x_1^2-x_2^2$.
Here, $N_1^{loc}$ plays the role of the 1-handle, but it is diffeomorphic to polygon with eight edges (see figure \ref{handledec}), whereas a usual 1-handle is diffeomorphic to $D^1 \times D^1$, which has four edges. For the last piece, suppose that the $1-$handle (in the usual handle-decomposition) is attached using an embedding
$$\phi\co S^0 \times [-\epsilon,\epsilon] \into S^1 = \partial N_0.$$
Then the last piece is
$$N_1^{triv} = [S^1 \setminus \phi(S^0 \times (-\epsilon/2,\epsilon/2))] \times [-1, 1].$$
This last piece has no analogue in a usual handle-decomposition; roughly, it fills in the rest of the space in $N$ after $N_0, N_2, N_1^{loc}$ are glued together. See figure \ref{handledec} for a picture of the pieces in the Milnor decomposition of $N = \bbR P^2$.
\\
\newline  Now consider the similar case when $\dim N =4$, and 
$f$ has three critical points $x_0, x_2, x_4$, with Morse indices $0,2,4$. 
Then there is a similar Milnor decomposition of $N$ with, for example, $N_0 = D^4 = N_4$.
To construct $N \subset E$ in this case (see \S \ref{ConstructionN}),  
we define several Lagrangian  manifolds $N_0 \subset E_0$, $N_4 \subset E_4$, 
$N_2^{loc} \subset E_2$, $N_2^{triv} \subset E_2 $ with boundary (with corners) which correspond to exactly to the pieces in the Milnor decomposition of $N$. Because of the way the fibers of the $E_i$ are constructed and the way the $E_i$ are  glued together, these $N_i$ glue together exactly as in the Milnor decomposition of $N$ (in particular, with the correct framings). From this it follows that the union $\cup_i N_i \subset E$ is smooth and diffeomorphic to $N$.

\subsection{Organization}

Here is a summary of the contents of this paper.

\begin{itemize}
\item[\S \ref{MorseBott}] We explain some techniques necessary for constructing the regular fiber $M$ and the vanishing spheres in $M$. (We only sketch the main ideas in this paper. See \cite{JA} for details.) The main technique involves attaching Morse-Bott type handles in the Weinstein category. This in turn is related to a generalization of Lagrangian surgery to the case where the Lagrangians intersect along any submanifold. This Lagrangian surgery construction is used to define the vanishing spheres, and it is also used in the construction of $M$ in general.

\item [\S \ref{fiber}] We explain how to construct the regular fiber $M$ 
and the vanishing spheres in $M$. We discuss the case when $f$ has three or four distinct Morse indices in detail, and we give a partial sketch of the general case.

\item[\S \ref{basics}] We review some basic constructions for symplectic Lefschetz fibrations. Already in this section we specialize to situations relevant to the case where $\dim N =4$,
and $f$ has only three distinct Morse indices $0,2,4$.

\item[\S \ref{pis}] We construct three  Lefschetz fibrations $\pi_i\co E_i \into D^2$, $i=0,2,4$
as discussed sketched in \S \ref{Esketch} above.

\item[\S \ref{pi}] We construct $(E,\pi)$ as the fiber connected sum of $(E_0, \pi_0)$, $(E_2, \pi_2)$ 
$(E_4, \pi_4)$, as sketched in \S \ref{Esketch} above.

\item[\S \ref{computingVC}] We check that the vanishing spheres of $(E,\pi)$ with respect to certain vanishing paths are indeed the expected Lagrangian spheres in $M$.

\item[\S \ref{ConstructionN}] We construct an exact Lagrangian embedding $N \subset E$, as  
sketched in \S \ref{Nsketch}, which satisfies the statements in Theorem $A$. (See Theorem \ref{main} and Proposition \ref{Prophomotopyequiv}.)

\item[\S \ref{3-manifolds}] We look at the case when $f\co N \into \bbR$ has Morse indices
$0,n,n+1,2n+1$, focusing on the case $\dim N =3$.
We explain the construction of $(E,\pi)$, and also the construction 
of $N \subset E$. 

\end{itemize}

\subsection{Motivations from symplectic topology} \label{motivations}
In the rest of this introduction we will attempt to motivate Problem \ref{complexify}, and our proposed solution Theorem $A$, from the point of view of symplectic topology. The most basic idea is that Lefschetz fibrations on a symplectic manifold $X$ give a non-unique description of the total space $X$ in terms of a regular hypersurface $Y= \pi^{-1}(b)$ and the vanishing spheres in $Y$ (analogous to a handle-decomposition in differential topology).  Thus, in principle, well-understood Lefschetz fibrations
for cotangent bundles should lead to insights about their symplectic topology. 
For more on this line of thought, see section \ref{bifibrations} below.
\\
\newline A more subtle and surprising fact is that Lefschetz fibrations on a symplectic manifold $X$ can also be used to analyze  arbitrary Lagrangian submanifolds $L \subset X$, and this is the more immediate source of our  interest in Problem \ref{complexify}. 
The original geometric idea, due to Donalson, is that $L \subset X$ should be some kind of combination of the Lefschetz thimbles, obtained by surgery operations. This has only been partially developed so far  in the form of matching paths and matching cycles, as in \cite{AMP}, \cite{HMS}, \cite[\S 16g]{S08} (see \S \ref{matchingpaths} for the definition and for some plans to develop this further). Nevertheless there is a rigorous algebraic version, due to Seidel \cite[Corollary 18.25]{S08}, which is formulated 
in 
the Fukaya category of $X$: 
\begin{itemize}
 \item[($\ast$)] \emph{Any $L \subset X$ can be expressed as a combination of 
  the Lefschetz thimbles by repeatedly forming mapping cones.}
\end{itemize}
Implicitly, this takes place in a context where ``mapping cone'' makes sense, namely the so-called \emph{derived} Fukaya category of $X$. (Conjecturally, mapping cones correspond to Lagrangian surgery, and so the algebraic and geometric view points should coincide.)
\\
\newline Theorem $A$ feeds into both of these (algebraic and geometric) ideas. First, and foremost,
it provides a good class of Lefschetz fibrations to be used in combination with Seidel's decomposition ($\ast$). In future work we will use this idea to study exact Lagrangian submanifolds $L \subset T^*N$, along the lines of \cite{S04} (see \S \ref{nearby} for more on this). Second,  Theorem $A$ helps  to develop Donalson's original geometric idea, because we indeed construct $N \subset E$
by doing surgery operations among the Lefschetz thimbles. (Actually, while this is essentially true, there is still some work to be done to relate what we do in this paper to that view point.)
Thus the construction of $N$ serves as a model for how to decompose any given Lagrangian $L \subset E$ fibering over some path in terms of the Lefschetz thimbles.
A closely related idea is to define \emph{generalized matching paths} for arbitrary manifolds. This involves going in the reverse direction: One starts with an embedded path $\gamma \subset \bbC$ which passes through several critical values; then, one  \emph{constructs} a Lagrangian  $L \subset E$ which fibers over $\gamma$, assuming that the Lefschetz thimbles lying over $\gamma$ satisfy suitable \emph{matching conditions}. We will explore both these ideas in future work  (see \ref{matchingpaths} for more details).
\\
\newline Finally, we mention there is a  conjecture of Seidel \cite[Remark 7.1]{S00II} which suggests a way to use Theorem $A$, Theorem $B$ in \cite{J}, and ($\ast$) to relate the two approaches of 
Fukaya-Seidel-Smith \cite{FSS,FSSB,S08} and Nadler-Zaslow \cite{NZ, N} for analyzing the Fukaya category of a cotangent bundle (the first approach being Picard-Lefschetz theory, and the second being comparison with constructible sheaves on $N$). We will not elaborate on this here, but we refer the reader to \cite{JB} for more on this.

\subsubsection{Bifibrations on cotangent bundles} \label{bifibrations} In the most optimistic view, one can 
start with a  Lefschetz fibration on $T^*N$ (or any symplectic manifold) and  
proceding inductively by introducing a new Lefschetz fibration on the fiber and, continuing in this way, reduce the symplectic topology of the total space to some combinatorial data. 
This strategy was successfully carried out very explicitly in the case of the quartic surface in \cite{HMS}, and then generalized to a more abstract general setting in \cite{S08}. In general, one 
often has very little detailed knowledge of the Lefschetz fibration, and it may be very complicated.
For example, if we are studying a closed symplectic 4-manifold $X$, the only thing to do in general is take a Donaldson pencil on $X$ (or maybe a variation on that which maps to $\bbC P^2$). Then, the combinatorial data one gets in this way cannot be reasonably handled, and in fact basic questions reduce to some hard  combinatorial group theory problems, as in \cite{A}. With this in mind, it is intriguing to start with one of our relatively simple and very explicit Lefschetz fibrations $\pi\co E \into D^2$ (and let us suppose for simplicity that $E \cong D(T^*N)$ as in remark \ref{pivsfC}). Then we can ask: Is there a similar Lefschetz fibration, say
$$\pi_2 \co M \into D^2$$
defined on the regular fiber? Naively, this seems plausible since $M$ is obtained, roughly speaking, by plumbing several disk cotangent bundles together (see \S \ref{Mgeneral}), and one would hope that the model complexifications on each of these disk cotangent bundles can be made to agree on the overlaps, so that they patch together to yield a fibration on $M$. (The actual construction, though, must combine the fibrations on each disk cotangent bundle in a more sophisticated way, 
by combining the regular fibers of each of the fibrations into one new regular fiber for the putative fibration on $M$.) In any case, once 
this is known one would like to extend this to the
slightly more sophisticated set up of a bifibration on $E$. Roughly, this is a holomorphic map $E \into \bbC^2$, with generic singularities, encoding a family of Lefschetz fibrations on the fibers of a Lefschetz fibration on $E$. (For the precise definition, see \cite[\S 15e]{S08}.) 
As pointed out to me by Maydanskiy, one potential application of such a bifibration (together with work in progress of Seidel) would be to construct exotic cotangent bundles along the same lines as  Maydanskiy's recent work on exotic sphere cotangent bundles \cite{Maksim}. 
More tentatively, such bifibrations (and similar structures on the fibers of $\pi_2$, etc.)
may lead to interesting matching relations among Lagrangian submanifolds in 
$E$, $M$, etc., in a spirit similar to \cite{HMS}, \cite{S08}. (See section \ref{matchingpaths} below for more about matching conditions which apply to  Lagrangian submanifolds more general than spheres.)

\subsubsection{Donaldson's decomposition and generalized matching paths} \label{matchingpaths}
Donaldson's idea is as follows. 
First, one assumes that $\pi$ maps $L$ onto an embedded path $\gamma$ such that $$f = \gamma^{-1} \circ (\pi|L) \co L \into [0,1]$$
is a Morse function, either by constructing a suitable $\pi$ for a given $L$ (as achieved in \cite{AMP}), or perhaps by deforming the given $L$ and $(E,\pi)$. Then, each critical point of $f$ is a critical point of $\pi$ lying on $L$, and each unstable and stable manifold of $f$ is part of a Lefschetz thimble of $\pi$.
The expectation  is that $L$ is isotopic to a surgery-theoretic combination of all these Lefschetz thimbles. This is well-understood when $L$ is a sphere and $\gamma$ runs between just two critical values: $L$ is then the union of two Lefschetz thimbles meeting at a common vanishing sphere and $\gamma$ is called a \emph{matching path}, see  \cite[\S 16g]{S08}, \cite{HMS}. 
\\
\newline As we mentioned above, the proof of Theorem $A$ involves constructing $N \subset E$ by doing successive surgery operations involving the Lefschetz thimbles, just as in Donaldson's proposed decomposition. More precisely, let us assume for convenience of notation that 
$f$ has just one critical point of each index. Then, we construct a sequence of (not neccesarliy closed) Lagrangian submanifolds $N_0, N_1, \ldots, N_m$, where $N_j$ is diffeomorphic to the $j$th \emph{sublevel} set of $f \co N \into \bbR$, i.e. $N_j \cong \{f \leq c_{j} - \epsilon \}$, where $\epsilon$ is small, $c_j$ is the $j$th critical value of $f$, and $N_{j+1} = N_j \# \Delta_j$ is obtained by a kind of Lagrangian surgery.
In future work we will develop two ideas suggested by this decomposition.
First, if $L\subset E$ is an arbitrary Lagrangian in the total space of an arbitrary Lefschetz fibration
then the above construction can serve as a model for how to decompose $L$ as a surgery-theoretic combination of the the Lefschetz thimbles, thus making Donalson's idea more precise.
Second, we will formulate the notion of a \emph{generalized matching path} for arbitrary Lagrangians $L$. 
\\
\newline The original notion of a mathing path gives a way of constructing 
Lagrangian spheres in the total space of a Lefschetz fibration.
To do this one assumes there is a path $\gamma\co [0,1] \into \bbC$ joining two critical values $c_1, c_2$ such that the two Lefschetz thimbles $\Delta_1, \Delta_2$ over $\gamma|_{[0,\frac{1}{2}]}$ and $\gamma|_{[\frac{1}{2}, 1]}$ have (Lagrangian) isotopic vanishing spheres $V_1, V_2$ in the fiber over $\gamma(\frac{1}{2})$ (see \cite[\S 16g]{S08}, or \cite[lemma 8.4]{AMP}). 
The generalization suggested by our proof of Theorem $A$ is roughly as follows. Take a path $\gamma$
joining several critical values $c_1, \ldots, c_m$, say $\gamma(t_j) = c_i$, $j =0,1\ldots, m$.
The simplest matching condition one could hope for would just involve Lefschetz thimbles of adjacent pairs of critical points: one would assume that the Lefschetz thimbles (up to isotopy) meet in 
the expected sphere given by Morse theory. However, there are important framing conditions missing here,
so the actual matching conditions will be inductive. Let $N_0$ be the the Lefschetz thimble of $c_0$, fibered over $\gamma|_{[t_0, \frac{1}{2}(t_{0}+t_1)]}$. For $j \geq 1$, assume inductively we have constructed  a manifold $N_{j-1}$ (built out of the Lefschetz thimbles corresponding to  $c_0, \ldots, c_{j-1}$).
Then the matching condition will involve $N_{j-1}$ and the Lefschetz thimble $\Delta_j$ at $c_i$
which fibers over $\gamma|_{[\frac{1}{2}(t_{j-1}+t_j), t_j]}$. 
First, the intersection of $\partial N_j$ and $\partial \Delta_j$ in the fiber must be a certain sphere, whose dimension is dictated by Morse theory. Second, this sphere will have a framing in $\partial N_j$ which comes from $\Delta_j$, and this must also be as dictated by Morse theory. 
Roughly, the framing works as follows. First, take a Weinstein neighborhood $D(T^*\Delta_j) \subset E$. Let $S_j$ denote the sphere $S_j = \partial \Delta_j \cap \partial N_j$ and assume that $S_j$ is 
bounded by a disk $U_j \subset \Delta_j$, where $\Delta_j \cong D^n$ and $U_j \cong D^k \subset D^n$.
(If we pretend $N$ exists for a moment, then $U_j$ is meant to be $\Delta_j\cap N$, which is part of the unstable manifold $U(x_j)$.)
Then the $k-$handle in $N$ corresponding to $x_j$ is represented by the disk conormal bundle 
$$D(\nu^*U_j) \subset D(T^*\Delta_j)$$
and the framing of the handle is encoded in the way $D(\nu^*U_j)$ meets $\partial N_{j-1}$.
(Here, the parameterization $\Delta_j \cong D^n$ and the Weinstein embedding $D(T^*\Delta_j) \subset E$ should be determined to a large degree by a  canonical (up to isotopy) parameterization of $\partial \Delta_j$, as in \cite[\S 16b]{S08}.)
\\
\newline A third point of interest is to compare the Donalson and Seidel decompositions.
Conjecturally, the mapping cone of a morphism between two Lagrangians $L_1, L_2$, let's say corresponding to a single point in $L_1 \cap L_2$, say $\alpha \in CF(L_1,L_2)$, 
is isomorphic to the Lagrangian surgery of $L_1$ and $L_2$,  say $L_1 \# L_2$:
$$Cone(\alpha\co L_1 \rightarrow L_2) \cong L_1 \# L_2,$$
and a version of this is known if $L_1$ is a Lagrangian sphere, see \cite[\S 17j]{S08}. It would be interesting to prove that $N_{j+1} = N_j \# \Delta_{j+1}$ is isomorphic to $Cone(N_j \rightarrow \Delta_{j+1})$ in the above matching path construction. (There is a corresponding result for the case of standard matching paths, \cite[lemma 18.20]{S08}.) This would show that whenever we have a generalized matching path with corresponding Lagrangian $L$, there is a Donaldson type decomposition of $L$ which coincides with a Seidel decomposition of $L$. 
Using this together with \cite{AMP}, for example, one might be able to prove a new version of Seidel's decomposition ($\ast$). This version  would rely on choosing different Lefschetz fibrations for different Lagrangians, rather than having one fixed Lefschetz fibration.

\subsubsection{Lagrangian submanifolds in $T^*N$} \label{nearby} 

Here we elaborate a little on how Theorem $A$ is relevant for the study of exact Lagrangian submanifolds $L \subset T^*N$ (see also the introduction to \cite{JA}). Our basic goal is to prove for certain $N$ that any closed exact Lagrangian submanifold $L \subset T^*N$ is Floer theoretically equivalent to $N$. This means in particular that $HF(L,L) \cong HF(N,N)$, so that $H^*(L) \cong H^*(N)$, and
$HF(L, T_x^*N) \cong HF(N, T_x^*N)$, so that $deg(L \into N) = \pm 1$.
Of course, results of this kind have been obtained for arbitrary manifolds $N$ in \cite{FSS, FSSB} and \cite{N, NZ}. We want to  consider a slightly different approach along the lines of the quiver-theoretic approach for the case $N = S^n$ in \cite{S04}. This approach  avoids spectral sequences and the use of gradings; thus it avoids one significant assumption on $L$, 
namely that it has vanishing Maslov class $\mu_L \in H^1(L)$.
\\
\newline To keep things concrete, take  $N = \bbC P^2$, and a Morse function $f\co N \into \bbR$ with three critical points  $x_0, x_2, x_4$,  with Morse indices 0, 2, 4. Let $(E,\pi)$ be the corresponding Lefschetz fibration from Theorem $A$, which models the complexification of $f$ on $D(T^*N)$. By construction, $\pi$ comes with  an explicit regular fiber $M$ and vanishing spheres $L_0, L_2, L_4 \subset M$. The main 
consequence of Theorem $A$ is that we have an exact Lagrangian embedding $N \subset E$.
Consequently there is an exact Weinstein embedding $D(T^*N) \subset E$. Now, let $L \subset T^*N$ be any closed exact Lagrangian submanifold. 
By rescaling $L \leadsto \epsilon L$ by some small $\epsilon >0$, we get an exact Lagrangian embedding $L \subset E$.
\\
\newline  Now that we know $L \subset E$ we can invoke Seidel's decomposition theorem ($\ast$).
Roughly, it says that we can represent $L$ algebraically (at the level of Floer theory) in terms of the Lefschetz thimbles $\Delta_4, \Delta_2, \Delta_0$ of $\pi$. To make this more explicit we 
need to know how the Lefschetz thimbles interact Floer theoretically. That is, we need to know
the Floer homology groups
\begin{gather}
 HF(\Delta_4,\Delta_2 ), \, HF(\Delta_2,\Delta_0), \, HF(\Delta_4,\Delta_0), \label{Floer}
\end{gather}
and also the triangle product (which is defined by counting holomorphic triangles with boundary on 
$\Delta_4, \Delta_2, \Delta_0$):
\begin{gather}
 HF(\Delta_4,\Delta_2 ) \otimes HF(\Delta_2,\Delta_0) \into HF(\Delta_4,\Delta_0). \label{triangle}
\end{gather}
These are precisely the calculations carried out in \cite{J}, except we actually  consider the vanishing spheres $L_i = \partial \Delta_i \subset M, i=0,2,4$  and do the corresponding equivalent  calculations in the regular fiber $M$. (In general, one does not expect to compute things like
(\ref{Floer}) and (\ref{triangle}) explicitly. It is only because of the very explicit 
and symmetrical nature of $M$ and $L_0, L_2, L_4$ that the calculations in \cite{J} can be carried out.)
\\
\newline The best way to phrase the answer is to think of a category $\mathcal{C}$ with three objects $\Delta_4, \Delta_2, \Delta_0$, where the morphisms and compositions are given by (\ref{Floer}) and (\ref{triangle}). Then Theorem $B$ in \cite{J} says 
that $\mathcal{C}$ is given by the following quiver with relations:
\begin{gather}\label{quiver}
\xymatrixcolsep{5pc}\xymatrix{\Delta_4 \ar@/^0.25pc/[r]^{a_1} \ar@/_0.25pc/[r]_{a_0} \ar@/^1.5pc/[rr]^{c_1} \ar@/_1.5pc/[rr]_{c_0} & \Delta_2 \ar@/^0.25pc/[r]^{b_1} \ar@/_0.25pc/[r]_{b_0} & \Delta_0} \\
b_1a_1 =0, \,  b_0a_0 = c_0, \,  b_0a_1 - b_1a_0 = c_1 \notag
\end{gather}
(More precisely, Theorem $B$ in \cite{J} says that $\mathcal{C}$ is isomorphic to another category, called the flow category, which is defined entirely in terms of  the Morse theory of $(N,f)$; that is where (\ref{quiver}) comes from.) The upshot of  Seidel's decomposition ($\ast$) in this case is that 
$L$ is represented  by a certain quiver representation of (\ref{quiver}):
\begin{gather}\label{rep}
\xymatrixcolsep{5pc}\xymatrix{W_2 \ar@/^0.25pc/[r]^{A_1} \ar@/_0.25pc/[r]_{A_0} \ar@/^1.6pc/[rr]^{C_1} \ar@/_1.5pc/[rr]_{C_0} & W_1 \ar@/^0.25pc/[r]^{B_1} \ar@/_0.25pc/[r]_{B_0} & W_0} \\
B_1A_1 =0, \, B_0A_0 = C_0, \, B_0A_1 - B_1A_0= C_1 \notag
\end{gather}
Here, the quiver representation (\ref{rep}) is just a choice of  vector-spaces  $W_4, W_2, W_0$ at each vertex, and a choice of linear maps $A_0, A_1, B_0, B_1,C_0, C_1$ satisfying the given 
relations. 
To show  $L$ is Floer theoretically equivalent to $N$ in  $T^*N$ is equivalent to showing that 
the representation (\ref{rep}) is necessarily isomorphic to the representation 
$$W_4 = W_2 = W_0 = \bbC, \, A_0= B_0 = C_0 = id, \, A_1= B_1 =C_1 =0.$$ 
(Of course, this is the representation corresponding to $N \subset T^*N$.)
The analogous problem for $N = S^n$ was solved in \cite{S04}. Work on this and related problems is currently in progress. 

\subsection*{Acknowledgements} 
The ideas about matching paths and Donaldson's decomposition have grown out of discussions I had with Denis Auroux a few years ago, while I was in graduate school. I thank him very much for his hospitality and for generously sharing ideas. The ideas about the nearby Lagrangian conjecture grew out of my Ph.D. work with Paul Seidel, and I thank him warmly as well. 

\section{Morse-Bott handle attachments and Lagrangian surgery}\label{MorseBott}
 
To construct the regular fiber $M$, we will use an extension  of Weinstein's handle attachment technique where we attach a Morse-Bott handle rather than a usual handle.
In this section we only explain the main ideas of this construction; for details we refer the reader to \cite{JB}.
\\
\newline  Recall that in \cite{W} Weinstein explains how to start with a Weinstein manifold $W=W^{2n}$
and attach a $k-$handle $D^k \times D^{2n-k}$, $k \leq n$, along an isotropic sphere in  the boundary of $W$
to produce a new Weinstein manifold $W'$.
(Recall that a Weinstein manifold  is an exact symplectic manifold $(W,\omega, \theta)$, $\omega = d\theta$, equipped with a Liouville vectorfield $X$  (i.e. one that satisfies $\omega(X, \cdot) = \theta$)  such that $-X$ points strictly inward along the boundary of $W$; in particular the boundary is of contact type.)
\\
\newline In \cite{JB} we extend this construction to a certain Morse-Bott case, namely where the handle is of the form
$$H = D(T^*(S^k \times D^{n-k})).$$
Here we think of $S^k \times \{0\} \subset H$ as the critical manifold and we think of 
$S^k \times D^{n-k} \subset H$ as the unstable manifold of $S^k \times \{0\}$.
It is not hard to describe how to attach $H$ to $W$ along the boundary in the \emph{smooth} category.
For that one needs two pieces of data: 
\begin{itemize}
 \item a submanifold 
$$S \subset \partial  W, \, S \cong S^k \times S^{n-k-1}$$ 
(where $S$ now plays the role of the attaching sphere), and
\item a bundle-isomorphism 
$$T^*(S^k \times D^{n-k})|_{(S^k \times \partial D^{n-k})} \into N_{\partial W}(S).$$
\end{itemize}
Here, $N_{\partial W}(S)= T(\partial W)|_{S}/T(S)$ is the normal bundle of $S$ in $\partial W$, 
and the bundle isomorphism determines a diffeomorphism (up to isotopy) 
from part of the boundary of $H$
to a neighborhood of $S$ in $\partial W$, 
$$\phi \co  D(T^*(S^k \times D^{n-k}))_{(S^k \times \partial D^{n-k})} \into U,$$
which we use to attach $H$ to $W$ to form $W' = W \cup H$. 
\\
\newline To extend this construction to the Weinstein category, we need only assume that
$S$ is Legendrian in $\partial W$. 
Then, Weinstein's construction can be modified so that one
starts with a Weinstein manifold $W$ and produces a new Weinstein manifold $W' = W \cup H$.
(See \cite{JB} for details.) The main point which is nontrivial is that the boundary of $W'$ is smooth and convex (i.e. transverse to $X$), and in particular of contact type. See figure \ref{figureM_0}
for a schematic picture of $W' = W \cup H$.
\\
\newline 
In the usual Weinstein handle attachment, $S$ is an isotropic sphere and the
normal bundle of $S$ in $\partial W$ can be decomposed as
$$N_{\partial W}(S) \cong \tau_S^{1} \oplus T^*S \oplus TS^{\omega}/TS,$$
where $\tau_S^{1}$ is the trivial real line bundle over $S$, and $TS^{\omega}$ is the symplectic orthogonal complement in $T(\partial W))$. Thus the first two terms necessarily sum to a trivial bundle, and the only  part which is possibly nontrivial is $TS^{\omega}/TS$ (denoted $CSN(S)$ in \cite{W}).
\\
\newline In our case, one has the same splitting 
\begin{gather}
 N_{\partial W}(S) \cong \tau_S^{1} \oplus T^*S \oplus TS^{\omega}/TS, \label{splitting}
\end{gather}
but, since $S$ is Legendrian, we have $TS^{\omega}/TS=0$. 
On the other hand $S \cong S^k \times S^{n-k-1}$ is not a sphere, so 
$$\tau_S^{1} \oplus T^*S \cong \tau_S^{1} \oplus T^*S^k \times T^*S^{n-k-1}$$ is  usually not trivial.
(Here  $T^*S^k \times T^*S^{n-k-1} \into S^k \times S^{n-k-1}$ is just the Cartesian product of the total spaces.) There is, however, a canonical isomorphism 
\begin{gather}
 \tau_S^{1} \oplus T^*S  \cong T^*(S^k\times D^{n-k})|_{S^k\times \partial D^{n-k}}. \label{bundle-iso}
\end{gather}
So, for us, we do not need to choose any framing data; we only need to choose the identification 
$S \cong S^k \times S^{n-k}$. See \S \ref{application} below for how this identification is chosen in some special situations.
\\
\newline
Note that we have only extended the Weinstein construction to a very particular Morse-Bott situation, namely the case where the critical manifold $C$ is a sphere, and the normal bundle of $C$ has a certain form. In general, a Morse-Bott function 
$f\co  X \into \bbR$ can have an arbitrary connected manifold $C$ as critical manifold, and the normal bundle of $C$ in $X$, say $E \into C$, can be arbitrary. In that situation the Morse-Bott handle
would be modeled on the bundle  $D(E_+) \times D(E_-) \into C$, with fiber $D^{k} \times D^{n-k}$, 
where $E \cong E_+ \oplus E_-$ is the splitting of $E$ into positive and negative eigenspaces of the Hessian of $f$ at $C$. It might be interesting to extend the Weinstein construction to the  general Morse-Bott case. 
\subsection{How this construction is applied}\label{application}

When we construct the regular fiber $M$ in various cases (see \S \ref{fiber}) we will repeatedly apply the above handle-attachment construction in the following set up.
We take $$W = D(T^*L),$$ 
i.e. the disk bundle cotangent bundle of some manifold $L=L^n$, with respect to some metric on $T^*L$.
Then we  take an embedded sphere
$$S^{n-k-1} \subset L$$
with a chosen parameterization of  tubular neighborhood of $S^{n-k-1} \subset L$,
\begin{eqnarray}
\phi\co  S^{n-k-1} \times D^{k+1} \into L \label{phi}  \label{framing}
\end{eqnarray}
corresponding to a chosen trivialization of the normal bundle of $S^{n-k-1}$ in $L$.
(This trivialization will be part of the framing data in  a chosen handle decomposition of our manifold $N$ corresponding to the Morse function $f\co  N \into \bbR$; $L$ will correspond to a regular level set of $f$ and $S^{n-k-1}$ will be an attaching sphere; see \S \ref{Mgeneral}.)
Thus the conormal bundle 
$$\nu^*S^{n-k-1} \subset T^*L$$
is trivial; we take $S \subset \partial W$ to be the sphere bundle
$$S = S(\nu^*S^{n-k-1}) \cong S^{n-k-1} \times S^{k}$$
with a corresponding trivialization $S \cong S^{n-k-1} \times S^{k}$ determined by the chosen framing (\ref{phi}). 
Of course $S$ is Legendrian in $\partial W = S(T^*L)$ since $\nu^*S^{n-k-1}$ is Lagrangian in $T^*L$.
The bundle isomorphism 
$$T^*(S^k \times D^{n-k})|_{S^k \times \partial D^{n-k})} \into N_{\partial W}(S)$$
is determined by (\ref{splitting}) and (\ref{bundle-iso}), since $S$ Legendrian 
implies $TS^{\omega}/TS=0$.

\subsection{Lagrangian surgery}\label{surgery}
One special property possessed by the Weinstein manifold 
$$W' = D(T^*N) \cup H$$ 
is the existence of an exact Lagrangian sphere $Z \subset W'$. Namely, $Z$ is the union 
of the disk conormal bundle $D(\nu^*S^{n-k-1})$ and the unstable manifold $S^k \times D^{n-k} \subset H$:
$$Z = D(\nu^*S^{n-k-1}) \cup (S^k \times D^{n-k}) \subset D(T^*N)\cup H.$$
Maybe it is  helpful to identify
$$Z = \{(u_1, \ldots, u_{n+1}) \in \bbR^{n+1} : \Sigma_i u_i^2=1 \};$$
then, one can think of $S^k \times D^{n-k} \subset H$ and  $D(\nu^*S^{n-k-1}) \subset D(T^*L)$ as corresponding to overlapping neighborhoods of the two subspheres
$$K_+= \{(u_1, \ldots, u_{k+1}, 0, \ldots, 0) \in \bbR^{n+1} : \Sigma_i u_i^2=1 \}, \text{ and}$$
$$K_- = \{(0, \ldots, 0,u_{k+2}, \ldots, u_{n+1}) \in \bbR^{n+1} : \Sigma_i u_i^2=1 \}.$$
$Z$ is smooth because the two pieces can be made to overlap smoothly;
it is Lagrangian since each piece is Lagrangian; and it is exact because each piece is exact, and the overlap region is connected. 

A more interesting fact is that $W'$ contains a Lagrangian submanifold 
$$L' \subset W'$$
which is diffeomorphic to the result of doing surgery on $L$ along the framed 
sphere $S^{n-k-1} \subset L$ (where the framing is (\ref{framing})).  This construction is used to define the Lagrangian  vanishing spheres in $M$, see \S \ref{fiber}. It is also used in the construction of $M$ in the general case, see \S \ref{Mgeneral}. 
One can think of $L'$ as the Lagrangian surgery of $L$ and $Z$ along $S^{n-k-1}$. 
(This construction can be generalized to the case of any two Lagrangians meeting 
cleanly along a connected closed manifold $C \subset L_1, L_2$, where 
$C$ has trivial normal bundle in $L_1$ and $L_2$, see \cite{JB}.)
\\
\newline To define $L'$ we start with an exact Weinstein embedding for $Z \subset W'$
$$\phi_Z\co  D_r(T^*S^n) \into W',$$
where  $D_r(T^*S^n)$ is the disk bundle with respect to the round metric of some suitably small radius $r>0$. Let us realize $T^*S^{n}$ as the following exact symplectic submanifold of $\bbR^{2n+2}$:
\begin{gather}
 T^*S^n = \{(u,v) \in \bbR^{n+1} \times \bbR^{n+1} : |u| =1, u\cdot v =0\}. \label{coordinates}
\end{gather}
Then, the subsphere $K_- \subset S^n$ we defined above has an obvious identification of its conormal bundle with $S^{n-k-1} \times \bbR^{k+1}$, because
$$\nu^*K_- = \{ ((0, \ldots, 0,u_{k+2}, \ldots, u_{n+1}), (v_1, \ldots, v_{k+1}, 0 \ldots, 0)) \in \bbR^{n+1} \times \bbR^{n+1}    : \Sigma_i u_i^2 =1  \}.$$
And of course there is a similar identification for $K_+$, 
$$\nu^*K_+ \cong S^k \times \bbR^{n-k}.$$
Now assume that $\phi_Z$ maps $D_r(\nu^*K_-)$ onto a neighborhood of
$S^{n-k-1} \subset L$. In fact we may assume $\phi_Z|_{D_r(\nu^*K_-)}$ agrees with the 
previously chosen framing (\ref{framing}):
\begin{gather}
\phi_Z|_{D_r(\nu^*K_-)} = \phi|_{ S^{n-k-1} \times D_r^{k+1} }: S^{n-k-1} \times D_r^{k+1}  \into L, \label{embeddingvsframing}
\end{gather}
where we are using the canonical identification $D_r(\nu^*K_-) \cong S^{n-k-1} \times D_r^{k+1}$.
(To see why we can assume (\ref{embeddingvsframing}), see the embedding (\ref{plumbingmap}) below; we can take our Weinstein embedding $\phi_Z$ to be the restriction of that. Alternatively, one can invoke Pozniack's local model for cleanly intersecting Lagrangians \cite[Proposition 3.4.1]{Poz}.)

\subsubsection{Construction of $L'$ up to homeomorphism (denoted $\widetilde L'$)} 
To see the rough idea for the construction of $L'$, let us assume for convenience that
$r=1$ for a moment. Now let $\Phi$ denote the time $\pi/2$ geodesic flow on 
$D(T^*S^n)= D_1(T^*S^n)$ (which is Hamiltonian).
The effect of $\Phi$ on $D(\nu^*K_+)$ is to fix vectors of zero length (i.e. points in $K_+$) and map the unit vectors $S(T^*K_+)$ diffeomorphically onto $S(T^*K_-)$, while vectors of intermediate length interpolate between these extremes. (See figure \ref{figsimpleVCdim2} for the case 
when $\dim D(T^*Z) =2$ and $\Phi$ is tweaked slightly to $\widetilde \Phi$.) Up to homeomorphism, $L'$ can be described as follows:
Define
$$T = \Phi(D(\nu^*K_+))$$ 
and set
$$\widetilde L'= (L \setminus \phi_Z( D(\nu^*K_-))) \cup \phi_Z(T).$$ 
Then it is clear that $\widetilde L'$ is homeomorphic to the surgery of $L$ along the framed sphere 
$S^{n-k-1} = \phi_Z(K_-)$, where the framing is given by (\ref{embeddingvsframing}).
\\
\newline The only problem is that $\widetilde L'$ is  not smooth, because 
$T$ is not tangent to $D(\nu^*K_-)$ along $\partial D(\nu^*K_-)$.
To fix this, we just need to tweak $\Phi$ slightly to get a new Hamiltonian diffeomorphism $\widetilde \Phi$ such that $\widetilde T = \widetilde \Phi(D(\nu^*K_+))$ agrees with $D(\nu^* K_-)$ in a neighborhood of $\partial D(\nu^* K_-)$. 
Then $L' = (L \setminus \phi_Z( D(\nu^*K_-))) \cup \phi_Z(\widetilde T)$ will be smooth.
(See figure \ref{figsimpleVCdim2} for a picture of 
the low dimensional situation: $D(T^*Z) \cong D(T^*S^n) = D(T^*S^1)$, $K_-, K_+ \cong S^0$.) 
We spell the details out now since we will need them available later.
\begin{figure}
\begin{center}
\includegraphics[width=3in]{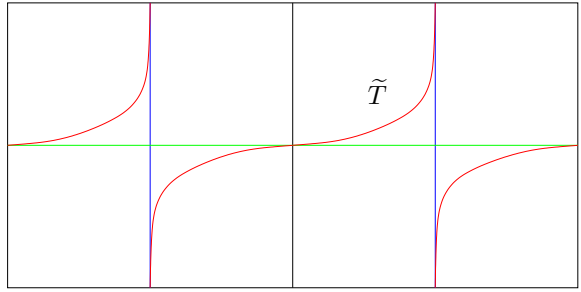} 
\put(-80,70){$\widetilde T$}
\caption{
Consider the low-dimensional situation $D(T^*Z) \cong D_1(T^*S^1)$, $K_+, K_- \cong S^0$.
We have depicted $D_1(T^*S^1)$ as $\bbR /2\pi\bbZ \times [-1,1]$.
The horizontal green line represents $Z \cong \bbR /2\pi\bbZ$; 
the two vertical blue lines represent $D(\nu^*K_-) \cong S^0 \times [-1,1]$; 
and the two curved red lines represent
$\widetilde T =  \widetilde \Phi (D(\nu^*K_+))$.}
 \label{figsimpleVCdim2}
\end{center}
\end{figure}
\subsubsection{Construction of $L'$ as a smooth Lagrangian submanifold} \label{surgeryh}
First, consider the \emph{normalized} geodesic flow on $T^*S^n \setminus S^n$,  which moves each (co)vector at unit speed for time $t$, regardless of its length. This has an explicit formula in terms of the coordinates (\ref{coordinates}): 
\begin{gather}
\sigma_t\co  T^*S^n \setminus S^n \into T^*S^n \setminus S^n, \notag
\sigma_{t}(u,v) = (\cos t u + \sin t \frac{v}{|v|}, \cos t \frac{v}{|v|} - \sin t u). 
\end{gather}

Given any function $H\co   T^*S^3 \setminus S^3 \into \bbR$ we let $\phi^H_t$ denote the time $t$ Hamiltonian flow of  $X_H$  (our convention is $\omega(\cdot, X_H) = dH$). It is elementary to check that 
for any $k \in C^{\infty}(\bbR,\bbR)$,
\begin{gather}
\phi^{k(H)}_t(p) =  \phi^H_{k'(H(p))t} (p) \label{k(H)}.
\end{gather}
Let
$$\mu\co  T^*S^3 \setminus S^3 \into \bbR, \mu(u,v) = |v|.$$
Then it is well-known that $\phi_t^{(1/2)\mu^2}$ is the usual geodesic flow and so (\ref{k(H)}) implies $\phi_t^{\mu}$ is equal to the \emph{normalized} geodesic flow, with the formula given by $\sigma_t$.
Now let  $h\co  \bbR \into \bbR$ be any smooth function satisfying 
\begin{gather}
h'(0) =0,\\ \notag
h'(t) = 1/2 , t \in [r/2,r],\\ \notag
h''(t) > 0, t \in [0,r/2),\\ \label{hconditions}
h(-t) = h(t) - t \text{ for small } |t| \notag
\end{gather}
In \S \ref{particularh} we will make a particular choice for $h$. Consider the map 
$$ F\co  D_r(T^*S^3)\setminus S^3 \into D_r(T^*S^3)\setminus S^3$$
defined by 
$$ F(u,v) = \phi_{\pi/2}^{h(\mu)}(u,v) =   \sigma_{h'(|v|)\pi}(u,v).$$
Then $ F$ extends continuously over the zero-section because $h'(0) =0$. To see that the extension is smooth one applies \cite[Lemma 1.8]{LES}. (This is why we need  $h(-t) = h(t) - t \text{ for small } |t|$.)
Call the extension 
$$\widetilde F \co  D_r(T^*S^3) \into D_r(T^*S^3).$$
(Here, $\widetilde F$  plays the role of $\widetilde \Phi$ before.)
Now define
$$\widetilde T = \widetilde  F(D_r(\nu^*K_+)).$$ 
Notice that 
$$\widetilde  F(D_{[r/2,r]}(\nu^*K_+)) = D_{[r/2,r]}(\nu^*K_-).$$
(One can see this by the formula for  $\sigma_{\pi/2}$.) 
It follows that
\begin{gather}
 \widetilde T \cap D_r(\nu^*K_-) = D_{[r/2,r]}(\nu^*K_-). \label{overlap}
\end{gather}
(This is an equality rather than just containment because $h''(t) > 0, t \in [0,r/2)$.)
Now set
$$L' = (L \setminus \phi_Z( D_{r/2}(\nu^*K_-))) \cup \phi_Z(\widetilde  T).$$
Then $L'$ is smooth because there is an overlap
$$[L \setminus \phi_Z( D_{r/2}(\nu^*K_-))] \cap \phi_Z(\widetilde T) = \phi_Z(D_{(r/2, r]}(\nu^*K_-)),$$
because of (\ref{overlap}). $L'$ is Lagrangian because $\widetilde T$ and $[L \setminus \phi_Z( D_{r/2}(\nu^*K_-))]$
are, and the overlap has nonempty interior in $L$. Since $\widetilde F$ Hamiltonian implies $\widetilde T$ is exact,
it follows that if $L$ is exact and the overlap region between $L$ and $\phi_Z(\widetilde T)$ (namely 
$\phi_Z(D_{[r/2,r]}(\nu^*K_-)) \cong S^{n-k-1} \times D^{k+1}_{[r/2,r]}$, $n\geq 1$, $k \geq 0$) is connected then $L'$ will be exact as well. 

\subsection{Plumbing} \label{plumbing}
There is an alternative construction called \emph{symplectic plumbing}, which (in particular) produces a manifold $W^0$ homeomorphic to $W' = D(T^*N) \cup H$ from the last section. We do not use this construction in this paper (mainly because the boundary of $W^0$ is not smooth), but it gives a useful alternative view point, and it is used in \cite{J}, so we discuss it briefly here. From time to time we may mention it to give some additional clarification in visualizing things.
\\
\newline Take two disk cotangent bundles $D(T^*L_1)$ and $D(T^*L_2)$ and assume that there is a closed manifold $K$ (connected, say) which has embeddings 
$$K \subset L_1, \text{ and } K \subset L_2.$$
Assume moreover that the normal bundle of $K$ in both $L_1$ and $L_2$ is trivial and choose
tubular neighborhoods
$$K \times D^{n-k} \subset L_1, \text{ and } K\times D^{n-k} \subset L_2,$$
where $\dim  L_i = n, \,  \dim  K =k$.
\\
\newline The idea of the symplectic plumbing construction is to 
glue  $D(T^*L_1)$ and $D(T^*L_3)$ together along neighborhoods of
$K \subset D(T^*L_1)$ and $K \subset D(T^*L_2)$, so that the intersection of $L_1$ and $L_2$ is precisely  $K$, and the intersection is \emph{clean}, or \emph{Morse-Bott}, i.e. $T(L_1) \cap T(L_2) = T(K)$. 
\\
\newline More precisely, the submanifolds $K \times D^{n-k} \subset L_1, L_2$ have tubular neighborhoods 
$$W_1 \subset D(T^*L_1), \, W_2 \subset D(T^*L_2),$$
with exact symplectomorphisms 
$$W_1, W_2 \cong D(T^*(K \times D^{n-k})).$$
(Here, we assume that $S(T^*(K \times D^{n-k})) \subset S(T^*L_i),$ $i=1,2$.)
To define the plumbing 
$$W^0 = D(T^*L_1) \boxplus D(T^*L_2)$$ 
we take the quotient of the disjoint union $D(T^*L_1) \sqcup D(T^*L_2)$, where we identify $W_1$ and  $W_2$ using a suitable exact symplectomorphism 
$$\eta\co  D(T^*(K \times D^{n-k}))\into D(T^*(K \times D^{n-k}))$$ 
which sends $K \times D^{n-k}$ to $D(\nu^*K)$ and  $D(\nu^*K)$ to $K \times D^{n-k}$.
This means that in $W^0$ a tubular neighborhood of $K$ in $L_1$ is identified with  
the disk conormal bundle of $K$ in $D(T^*L_2)$, and vice-versa. (This condition is motivated 
by Pozniack's local model \cite[Proposition 3.4.1]{Poz}.)
\begin{figure}
\begin{center}
\includegraphics[width=3in]{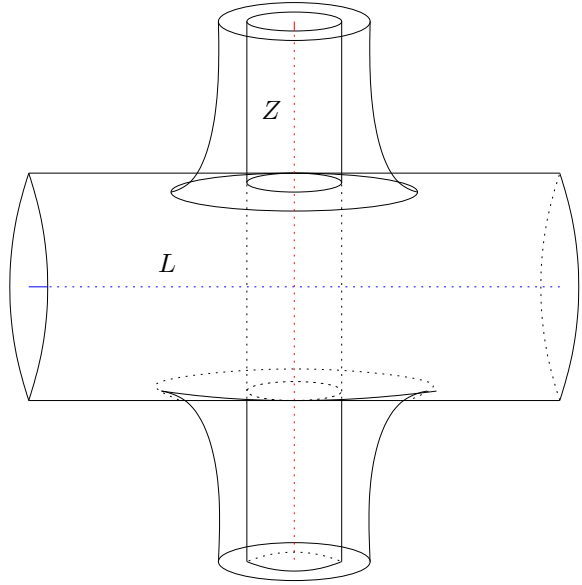} 
\put(-121,175){$Z$}
\put(-160,117){$L$}
\caption{Schematic of $W^0 = D(T^*L) \boxplus D(T^*S^n)$ embedded into $W'= D(T^*L) \cup H$  near plumbing or handle-attachment region. Parts of $L$ and the Lagrangian sphere $Z$ 
are also labeled.}
\label{figureM_0}
\end{center}
\end{figure}
\\
\newline To define $\eta$, let us pass for a moment to the noncompact model 
$$T^*(K \times \bbR^{n-k}) \cong T^*K \times T^*\bbR^{n-k} \cong T^*K \times \bbC^{n-k}.$$
It is easy to see that $\nu^*K \subset T^*(K \times \bbR^{n-k})$ corresponds to 
$K \times i\, \bbR^{n-k} \subset T^*K \times \bbC^{n-k}$ in this model.
Thus, we can $\eta$ to be the restriction of the map 
$$id_{T^*K} \times m(i)\co  T^*K \times \bbC^{n-k} \into T^*K \times \bbC^{n-k}.$$
There is one sticky point, which is that $W_1$ and $W_2$
correspond to subsets of 
$T^*K \times \bbC^{n-k}$ with boundary (with corners), and so one has to be a little bit careful to choose the disk bundles $D(T^*L_1)$ and $D(T^*L_2)$ so that these boundaries correspond nicely under the map $id_{T^*K} \times m(i)$. (See \cite{JB} for details.)
\\
\newline To relate this to the handle attachment $W' = D(T^*L) \cup H$ we take 
$L_1 =L$ to be any manifold and $K = S^{n-k-1} \subset L$ with the chosen framing
$S^{n-k-1} \times D^{k+1} \subset L$ as in \S \ref{application}. Then we we take 
$$L_2 = S^n=\{(u_1, \ldots, u_{n+1}) \in \bbR^{n+1} : \Sigma u_i^2 =1  \},$$ 
and we take $S^{n-k-1} =  K_- \subset L_2$ as in \S \ref{surgery}
with the the obvious canonical framing $S^{n-k-1} \times D^{k+1} \subset L_2$.
Then, 
$$W^0 = D(T^*L) \boxplus D(T^*S^n)$$
is homeomorphic to 
$$W' = D(T^*L) \cup H.$$
Moreover, there is an exact symplectic embedding
\begin{gather}
\rho\co  D(T^*L) \boxplus D(T^*S^n)\into D(T^*L) \cup H \label{plumbingmap}
\end{gather}
such that $\rho|_{L} = id_{L}$ and $\rho(S^n) = Z$ (see figure \ref{figureM_0}). See \cite{JB} for the proof. 

\section{Construction of the regular fiber $M$ and the Lagrangian vanishing spheres} \label{fiber}
Let $N$ be a closed manifold and let $f\co  N \into \bbR$ be self-indexing Morse function.
In this section we will explain how to construct the regular fiber $M$ and the Lagrangian vanishing spheres $L_1, \ldots, L_m \subset M$ for $\pi\co  E \into D^2$. 
\\
\newline We will deal with three cases:
\begin{enumerate}
\item $f$ has three distinct Morse indices $0,n,2n$ (see \S \ref{4mfds} and \ref{2nfiber}). \label{three}
\item $f$ has four distinct Morse indices $0,n,n+1, 2n+1$ (see \S \ref{3mfds}). \label{four}
\item The general case (partial sketch- see \S \ref{Mgeneral}). \label{general}

\end{enumerate}

The construction of $M$ in each case is identical as the dimension of $N$ varies.
For this reason we will keep things slightly more concrete in the first two cases above by focusing 
on the cases when $\dim  N =4$ and $\dim  N =3$ respectively. See \S \ref{2nfiber} for how things work in an arbitrary dimension in case (\ref{three}).

\subsection{Constructing $M$ and the vanishing spheres in case (\ref{three}), $\dim N =4$} \label{4mfds}


Suppose $N$ is a closed 4-manifold and 
$$f\co  N \into \bbR$$ 
has critical points  $x_0, x_2^{\, j}, x_4$, $j=1, \ldots, k$, where the subscript indicates the Morse index. Let $g$ be a Riemannian metric such that $(f,g)$ is Morse-Smale.
\\
\newline First,  $(f,g)$ induces a handle decomposition of $N$, which determines $k$ framed knots $K_j \subset S^3$, which are the attaching spheres of the 2-handles, together with a parameterization of a tubular neighborhood of each $K_j$ 
\begin{gather}
 \phi_j\co  S^1 \times D^2 \into S^3, \label{framingj}
\end{gather}
determined by the framing for the 2-handle up to isotopy. 
Set 
$$L_0 = S^3.$$
We start with the disk bundle $D(T^*L_0)$.
To construct $M$, we attach a Morse-Bott handle
$$H_j = D(T^*(S^1 \times D^2))$$
to $D(T^*L_0)$ for each $j$, where the gluing region is a neighborhood
of $S(\nu^*K_j)$ in $S(T^*L_0)$, and the framing $\phi_j$ determines the gluing map. 
This produces a Weinstein manifold
$$M = D(T^*L_0) \cup (\cup_j H_j),$$
as we explained in \S \ref{application}. We have for each $j$
an exact Lagrangian 3-sphere 
$$L_2^j \subset M$$
which is the union of $D(\nu^*K_j)$ and $S^1 \times D^2 \subset H_j$. 
(Each $L_2^{\, j}$ corresponds to what we called $Z$ in \S \ref{surgery}.) See figure \ref{figureM_0} in \S \ref{plumbing}  for a schematic picture of the region near each attaching region.
\\
\newline Now define $L_4$ as the Lagrangian surgery of $L_0$ and all the $L_2^j$'s. In \S \ref{surgery}
we explain how to define the Lagrangian surgery of $L \subset D(T^*L)\cup H$ and a single Lagrangian sphere $Z \subset D(T^*L) \cup H$. We use that definition for all $L_2^j$'s simultaneously, where each $L_2^j$ plays the role of $Z$.
$L_4$ is exact since it is simply-connected.  Thus we have defined exact Lagrangian spheres
$L_0, L_2^{\, j}, L_4$ in $M$, one for each critical point  $x_0, x_2^{\, j}, x_4$ of $f$.
 If $\dim N =2$ then there is an analogous construction of a 2 dimensional version of $M$; see figure \ref{Fiberofpi0} in \S \ref{pi4} for the case when $f$ has 
four critical points with Morse indices 0,1,1,2.
\\
\newline There is one ingredient in the surgery construction which is useful to record here. Namely,
we must fix exact Weinstein embeddings for each $L_2^j \subset M$,
$$\phi_{L_2^{\, j}}\co  D_r(T^*S^{3}) \into M,$$
where  $D_r(T^*S^{3})$ is the disk bundle with respect to the round metric of radius $r>0$.
$\phi_{L_2^{\, j}}$ should also agree with the framing (\ref{framingj}) along $D_r(\nu^*K_-)$, that is, we assume 
$$\phi_{L_2^{\, j}}|_{D_r(\nu^*K_-)}\co  D_r(\nu^*K_-) \into L_0$$
coincides with 
$$\phi_j|_{S^1\times D^2_r}\co  S^1\times D^2_r  \into S^3.$$ 

\begin{remark} \label{exactnessdim2}In the case $\dim N=2$ one can show that the analogue of $L_4$ (which would be $L_2$ corresponding to a critical point of index 2) is exact by applying lemma \ref{halftwist}. That lemma says $L_4 =\tau(L_4')$ where $L_4'$ is exact and $\tau$ is exact. The analogous lemma in the case $\dim N =2$ says $L_2 =\tau(L_2')$.
\end{remark}
 
\subsection{Constructing $M$ and the vanishing spheres in case (\ref{three}), $\dim N =2n$} \label{2nfiber}

In this section we  quickly sketch how the construction  works in the more general case where $\dim N =2n$ and the Morse indices of $f$ are $0,n,2n$; it is much the same as \S \ref{4mfds}.
\\
\newline In this case the handle-decomposition of $N$ corresponding to $(f,g)$ determines 
$k$ attaching spheres 
$$K_j \subset S^{2n-1}, \, K_j \cong S^{n-1}$$
with framings
$$\phi_j \co  S^{n-1} \times D^{n} \into S^{2n-1}.$$
To construct $M$, we set $L_0 = S^{n-1}$
and then attach $k$ Morse-Bott handles 
$$H_j \cong D(T^*(S^{n-1} \times D^n))$$ 
to $D(T^*L_0)$ where the attaching region is a neighborhood of $S(\nu^*K_j) \subset S(T^*L_0)$,
and the attaching maps are determined by framings $\phi_j$. The other vanishing cycles $L_n^j, L_{2n}$ are defined as before.

\subsection{Constructing $M$ and the vanishing spheres in case (\ref{four}), \dim N = 3} \label{3mfds}

In this section we explain how, for a closed 3-manifold $N$, one can construct
the regular fiber $M$ and vanishing spheres $L_0, L_1,L_2,L_3$ in $M$. 
This discussion applies equally well to self-indexing Morse functions
$f\co  N \into \bbR$ with four critical values $0,n,n+1,2n+1$. 
See section \ref{2nfiber} to see how things are much the same from one dimension to the next.
\\
\newline Let $(N,f,g)$ be a triple consisting of a closed 3-manifold $N$, and a self-indexing Morse function $f\co  N \into \bbR$, together with a Morse-Smale metric $g$ on $N$.
Then $(N,f,g)$ determines a Heegard diagram for $N$ as follows.
Let 
$$T = f^{-1}(3/2).$$ 
This  is a closed 2-manifold of some genus $h$. 
We may assume $f$ has $h$ critical points of index 
$1$ and $2$, $x_1^{\, j}$, $x_2^{\, j}$, $j=1,\ldots, h$.
Then there are $h$ circles
$$\alpha_j, \beta_j \subset T, j=1, \ldots g,$$
namely 
$$\alpha_j = S(x_1^{\, j}) \cap T, \beta_j = U(x_2^{\, j}) \cap T,$$
where $S(x_1^{\, j})$ is the stable manifold of $x_1^{\, j}$ and $U(x_1^{\, j})$ is the unstable manifold of $x_1^{\, j}$.
The data of $T$ together with the circles $\alpha_j,\beta_j$ is called a Heegard diagram for $N$; it determines the diffeomorphism type of $N$. 
In addition, $(f,g)$ determine framings 
$$\phi_j^\alpha \co  S^1 \times D^1 \into T$$
$$\phi_j^\beta\co  S^1 \times D^1 \into T$$
for $\alpha_j$, $\beta_j$ respectively. (Of course, in this dimension, there are only two possible framings for each $\alpha_j$ or $\beta_j$ and they give rise to diffeomorphic manifolds. But in higher dimensions the analogue of these framings are important.)
\\
\newline Then $M$ is defined as follows.
Set 
$$L_1^{\, j} , L_2^{\, j} = S^2,j = 1, \ldots, h.$$ 
Now consider the disk bundle $D(T^*T)$ and 
consider the disk conormal bundles  $D(\nu^*(\alpha_j))$ and $D(\nu^*(\beta_j))$,  each being diffeomorphic to $S^1 \times D^1$. Since $(f,g)$ is Morse-Smale, it follows that $\alpha_j$ and $\beta_k$ are transverse for any $j,k$.  Therefore we may assume that 
the boundaries of the disk conormal bundles $S(\nu^*(\alpha_j))$ and $S(\nu^*(\beta_k))$
are disjoint for every $j,k$. This means we can attach handles to $D(T^*T)$ along $S(\nu^*(\alpha_j))$ and $S_R (\nu^*(\beta_k))$ as follows. Take $2h$ Morse-Bott handles 
$$H_j^\alpha = D(T^*(S^0 \times D^2))$$ 
$$H_j^\beta = D(T^*(S^0 \times D^2)).$$ 

\begin{figure}
\begin{center}
\includegraphics[width=2.5in]{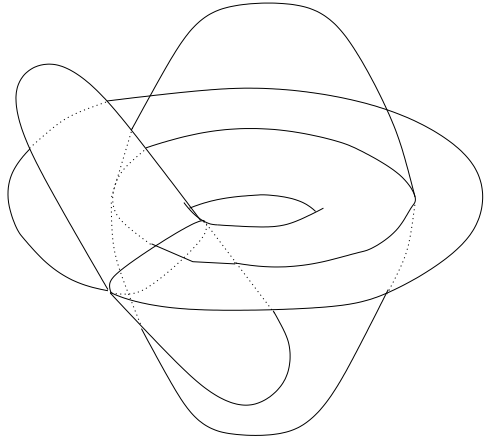} 
\caption{The case $N= S^3$, where $f$ has four critical points of index 0,1,2,3. This is a schematic of $M$, 
depicting $T=T^2$, with one $\alpha$ curve and one $\beta$ curve, together with two vanishing spheres, $L_1$  and $L_2$ which meet $T$ at $\alpha$ and $\beta$. The other two vanishing spheres $L_0$ and $L_2$ are not depicted; they are obtained as the surgery of $T$ and $L_1$, respectively $T$ and $L_2$. (The caption of figure \ref{VC3mfd} attempts to describe how to visualize $L_0$ and $L_3$.)}
\label{MVCfor3mfd} 
\end{center}
\end{figure}
\begin{figure}
\begin{center}
\includegraphics[width=2in]{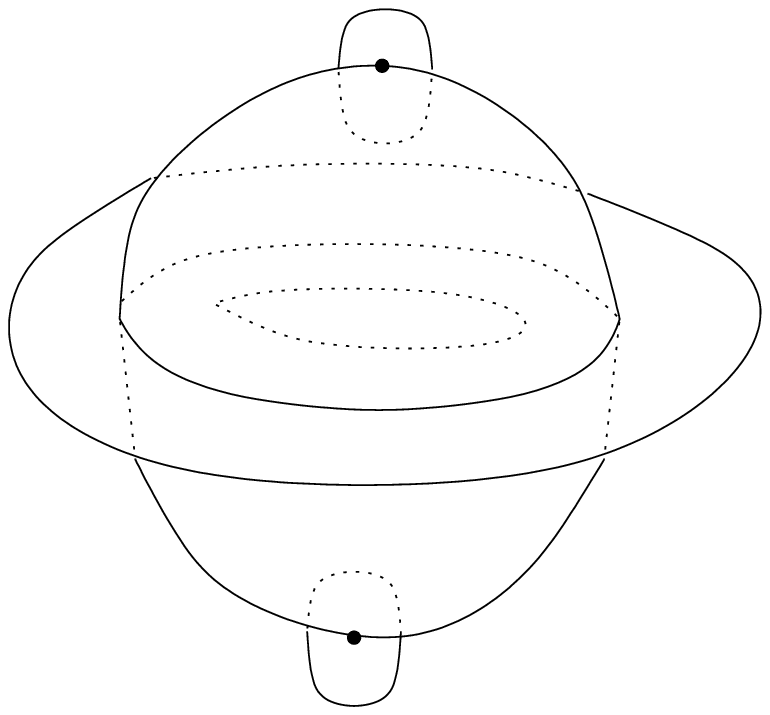} 
\caption{A two dimensional vanishing sphere (relevant for $\dim N =3$). Depicted schematically are the (disk) conormal bundles of $K_+ = S^0$ and $K_- = S^1$. To visualize the surgery of $T$ and $L_1$, say, first imagine bending the conormal bundle of $K_+$ so that its two boundary circles are identified with the two boundary circles of the conormal bundle of $K_-$. (Formally speaking, this ``bending''  is done by reparameterized geodesic flow.) Second, imagine that the conormal bundle of $K_-$ is identified with a neighborhood of the $\alpha$ curve in $T$, say $N(\alpha)$ (see figure \ref{MVCfor3mfd}). Thus, the union of the bent conormal bundle of $K_+$ and $T \setminus N(\alpha)$ together form a 2-sphere which is $L_0$ .}
\label{VC3mfd}
\end{center}
\end{figure}

(Note that, when $\dim N =3$ as in our case, each $D(T^*(S^0 \times D^2))$  is just the disjoint union of two usual (not Morse-Bott) $2-$handles $D(T^*D^2)$.)
To construct $M$ we attach each $H_j^\alpha$ and $H_j^\beta$ to the boundary of $D(T^*T)$
in such a way that the core of $H_j^\alpha$, that is $S^0 \times D^2$, is glued to 
$D(\nu^*(\alpha_j))$ along their boundaries, and similarly for $H_j^\beta$ and 
$D(\nu^*(\beta_j))$. Thus the union of the core of $H_j^\alpha$, given by $S^0 \times D^2$, and $D (\nu^*\alpha_j))$ forms an exact Lagrangian 2-sphere
$$L_1^{\, j} \subset M$$
which intersects $T$ in $\alpha_j$, and similarly we have
$$L_2^{\, j} \subset M$$
which intersects $T$ in $\beta_j$. Here, $L_1^{\, j}, L_2^{\, j}$ are analogous to $Z$ in \S \ref{surgery}. (If one is not concerned about $M$ having a smooth boundary, one can alternatively define $M$ as the \emph{plumbing} (see \S \ref{plumbing})  of $D_R(T^*T)$ and $D_r(T^*L_1^{\, j})$ along $\alpha_j$ and $D_r(T^*L_2^{\, j})$ along $\beta_j$, for some $0< r< R$.)
The precise attaching maps for $H_j^\alpha$ and $H_j^\beta$ are determined by the 
framings $\phi_j^\alpha$ and $\phi_j^\beta$. 
Let 
$$K_- = \{(u_1, u_2, 0) : \Sigma_j u_j^2 =1 \} \subset S^2$$
$$K_+ = \{(0,0,u_3) : u_3 = \pm1\} \subset S^2.$$
There are exact Weinstein embeddings of $L_2^{\, j}$ and $L_1^{\, j}$,
$$\phi_{L_2^{\, j}}\co  D_r(T^*S^3) \into M$$
$$\phi_{L_1^{\, j}}\co  D_r(T^*S^3) \into M,$$
such that 
$$\phi_{L_1^{\, j}}(K_-) = \alpha_j$$
$$\phi_{L_2^{\, j}}(K_-) = \beta_j$$
and 
$$\phi_1^{\, j}|_{D_r(\nu^*K_-)} = \phi_j^\alpha$$
$$\phi_2^{\, j}|_{D_r(\nu^*K_-)} = \phi_j^\beta.$$
Here $T^*L_1^{\, j}$, $T^*L_2^{\, j}$ are equipped with the standard round metric of $T^*S^2$.
We define $L_0$ as the Lagrangian surgery of 
$T$ and the $L_1^{\, j}$'s, as in \S \ref{surgery}. More precisely, there are two parts to $L_0$: one is
$T \setminus (\cup_j \phi_j^\alpha(S^1 \times D_{r/2}^1))$, and the other is 
the union of some subsets of 
$\phi_{L_1^{\, j}}(D_r(T^*S^2))$ which are defined by taking $D_r(\nu^*K_+) \cong S^0 \times D^2$ and applying a reparameterized geodesic flow map to it. 
$L_3$ is defined similarly: it is the surgery of $T$ and the $L_2^{\, j}$'s.
See figure \ref{MVCfor3mfd} for a schematic of $M$, and figure \ref{VC3mfd} for a schematic of a vanishing sphere to help with visualizing surgery. 

\subsection{Partial sketch of the construction of $M$ for arbitrary Morse functions.}\label{Mgeneral}

Let $N$ be a closed manifold of dimension $n$. Assume for simplicity of notation that $f\co  N \into \bbR$ is self-indexing with just one critical point of each index  $i_0,i_1, \ldots, i_k$.
As a starting motivation, we expect to have one Lagrangian sphere $L_j$ in $M$ for each critical point $x_{i_j}$ of $f$. Secondly, in $M$ we expect to find that each regular level set of $f$ in $N$ is embedded as a Lagrangian submanifold in $M$. This is because in the total space of the Lefschetz fibration $f_\bbC\co  D(T^*L_0) \into \bbC$ the real level set $f^{-1}(i_j \pm \epsilon)$ is obviously Lagrangian in the complex level set $f_\bbC^{-1}(i_j\pm \epsilon)$; but then we can transport all of these Lagrangians into one common fixed regular fiber of $f_\bbC$.
\\
\newline Here then is the rough idea for the construction of $M$. Take $D(T^*L_0)$, \ldots, $D(T^*L_k)$, one for each critical point of $f\co  N \into \bbR$. We take a handle-decomposition of $N$ corresponding to $f$ (and some fixed metric $g$ such that $(f,g)$ is Morse-Smale). Let the regular level sets of $f$ be denoted 
$$N_j = f^{-1}(i_j+\epsilon),\, j =0, \ldots, k.$$
And denote the corresponding sublevel sets by 
$$N_j^{\leq} = \{x \in N: f(x) \leq f(i_j + \epsilon)\}.$$
We will construct a sequence of manifolds $M_0, \ldots, M_{k-1}$ where $M_j$ is in fact a model for
the regular fiber of the complexification of $f$ restricted to the $j$th sublevel set, $f|_{N_j^{\leq}}: N_j^{\leq} \into \bbR$. 
In particular, $M= M_{k-1} = M_k$ is a model for the regular fiber of the complexification of $f$ on $N =N_k$. (It happens that $M_{k-1}$ is the regular fiber for the complexification of both $f|_{N_{k-1}^{\leq}}$ and $f|_{N_{k}^{\leq}}$ = $f_{N}$.) 
\\
\newline Since $N$ is a closed manifold we know that $N_0$ and $N_{k-1}$ are both spheres $S^{n-1}$, 
and in fact we identify $L_0 = N_0$ and $L_{k} = N_{k-1}$.
Let $M_0 = D(T^*L_0)$. Let $S^{i_1-1} \subset N_0=L_0$ be the framed attaching sphere of the first handle (of index $i_1$). Let $D(\nu^*S^{i_1-1}) \subset D(T^*L_0)$ denote the conormal bundle of $S^{i_1-1}$. To get $M_1$, attach a Morse-Bott handle to $M_0$ along $\partial D(\nu^*S^{i_1-1})$. (The result is homeomorphic to the plumbing of $D(T^*L_0)$ and  $D(T^*L_1)$.)
Now, as in \S \ref{surgery}, we have inside $M_1$  the Lagrangian surgery of $L_0$ and $L_1$. Since this is diffeomorphic to the result of doing surgery on $L_0$ along the framed sphere $S_{i_1-1}$, we identify it with the next level set $N_1$. Now the next framed attaching sphere $S^{i_2-1}$ is a subset of $N_1$. Take an exact Weinstein embedding for $N_1$, say 
$$D(T^*N_1) \subset M_1,$$ 
and look at $D(\nu^*S^{i_2-1}) \subset D(T^*N_1)$. The key technical obstruction to proceding at this point is the following: We must check that the boundary of $D(\nu^*S^{i_2-1})$ reaches the boundary of the ambient space $M_1$:
$$\partial D(\nu^*S^{i_2-1}) \subset \partial M_1.$$ 
If this is satisfied then we can attach a Morse-Bott handle to $M_1$ along 
$\partial D(\nu^*S^{i_2-1})$ to get $M_2$. The union of the core of this handle attachment and $D(\nu^*S^{i_2-1})$ will form $L_2$. (Alternatively, up to homeomorphism, one can think 
of $M_2$ as defined by plumbing $D(T^*L_2)$ onto the subset $D(T^*N_1) \subset M_1$.)
Then, as in \S \ref{surgery},   we have inside $M_2$ the Lagrangian surgery of $N_1$ and $L_2$, which is the next level set $N_2$, etc $\ldots$ We continue in this way, where at each stage to get $M_{j+1}$ from $M_{j}$ we do a Morse-Bott handle attachment. (Or equivalently, up to  homeomorphism, we plumb on a copy of $D(T^*L_{j+1})$ onto the Weinstein neighborhood $D(T^*N_j) \subset M_j$). At the last stage $N_{i_k-1}$ is equal to the last vanishing sphere $L_k$, and we stop.
\\
\newline Each time we do the Morse-Bott handle attachment to go from $M_j$ to $M_{j+1}$ 
we must first check that the conormal bundle $D(\nu^*S^{i_j-1})$ in the Weinstein neighborhood $D(T^*N_j)\subset M_j$ satisfies $$\partial D(\nu^*S^{i_j-1}) \subset \partial M_j.$$
But checking this condition turns out to be not completely straight-forward because of how the Lagrangians are all twisted up. The full treatment of this general case is therefore postponed to a future paper.

\section{Basics of symplectic Lefschetz fibrations} \label{basics}

In this section we review without proof some basic facts and constructions in the theory of  symplectic Lefschetz fibrations. The material is mostly taken from \cite[Ch.1]{LES} and \cite[\S 15,16]{S08}; the reader can find details there.

\subsection{The local model $q\co   \bbC^{n+1} \into \bbC.$} \label{localmodel}
The basic local model for Lefschetz fibrations is
$$q\co   \bbC^{n+1} \into \bbC, \, q(z) = \Sigma_j z_j^2.$$
Realize $T^*S^n$ as 
$$\{(u,v) \in \bbR^{n+1} \times  \bbR^{n+1} \co  |u| =1, u\cdot v=0\}.$$ 
For each $s>0$, there is a canonical exact symplectic isomorphism
$$\rho_s\co  q^{-1}(s) \into T^*S^n, \, \rho_s(z) = (x|x|^{-1}, -y|y|).$$
That is, $\rho_s$ satisfies 
$$\rho_s^*(\Sigma_j -u_jdv_j) = \Sigma x_j dy_j.$$
Define
$$\rho_0\co  q^{-1}(0) \setminus \{0\} \into T^*S^n \setminus S^n$$
by the same formula 
$$\rho_0(z) = (x|x|^{-1}, -y|y|).$$
Then this also gives an exact symplectomorphism.
\\
\newline Recall we have \begin{gather}
\sigma_t\co  T^*S^n \setminus S^n \into T^*S^n \setminus S^n, \notag
\sigma_{t}(u,v) = (\cos t u + \sin t \frac{v}{|v|}, \cos t \frac{v}{|v|} - \sin t u). \label{sigma}
\end{gather}
This  is the \emph{normalized} geodesic flow on $T^*S^n \setminus S^n$,  
which moves each (co)vector at unit speed for time $t$, regardless of its length. 
For any $w \in D^2$, let 
$$\Sigma_w = \sqrt{w}S^n \subset \bbC^{n+1}, \Sigma = \cup_w \Sigma_w.$$
There is a fiber preserving diffeomorphism, which is a fiber-wise exact symplectomorphism  
$$\Phi\co  \bbC^{n+1} \setminus \Sigma \into (T^*S^n \setminus S^n)\times \bbC,$$ 
$$\Phi(z) = (\sigma_\theta (\rho_s(e^{-i\theta}z)), q(z)),$$
where $q(z) = se^{i\theta}.$
Let
$$k\co  \bbC^{n+1} \into [0,\infty), k(z) = |z|^4 -|q(z)|^2.$$
Then $\Sigma = \{k=0\}$ and for any $r>0$ and any $t \geq r$ we have
$$\Phi^{-1}(S_{r}(T^*S^3)) = k^{-1}(4r^2) \subset \bbC^{n+1},\text{ and }\Phi^{-1}(D_{[t, r]}(T^*S^3)) = k^{-1}([4t^2,4r^2]).$$
Here,  
$$S_r(T^*S^3) = \{(u,v) \in T^*S^3 : |v| = r\}, \, 
D_{[t, r]}(T^*S^3) = \{(u,v) \in T^*S^3 : |v| \in [t, r]\}.$$

$q$ has well-defined symplectic parallel transport maps between any two regular fibers, where the
connection is given by the symplectic orthogonal to the fibers, 
$$T_z(q^{-1}(w))^{\omega} = \overline{z} \,\bbC.$$ 
(Indeed, parallel transport preserves the level sets of $k$, so it is well-defined.)
For any  embedded path 
$$\gamma\co  [0,1] \into D^2 \text{ with } \gamma(0) = w \neq 0, \gamma(0) =0$$ 
the corresponding \emph{Lefschetz thimble} (or \emph{vanishing disk}) is 
$$\Delta_\gamma = \cup_t \Sigma_{\gamma(t)}$$
and the \emph{vanishing sphere} in $q^{-1}(w)$ is 
$$V_\gamma = \Sigma_{w} \subset q^{-1}(w).$$

$\Phi$ can be described more conceptually as follows. 
\begin{lemma} \label{radial} Consider the map 
$$\widetilde \Phi\co  \bbC^{n+1} \setminus \Sigma \into q^{-1}(0) \setminus \Sigma_0$$
given by the radial symplectic parallel transport maps from $q^{-1}(w) \setminus \Sigma_w$ to $q^{-1}(0) \setminus \Sigma_0$. 
Then $\Phi = \rho_0 \circ \widetilde \Phi.\,  \square$
\end{lemma} 
(We omit the proof. One can check this using the rotational symmetry of the radial transport map as explained in the proof of lemma 1.10 in \cite{LES}.)
\\
\newline For any $\theta \in [0,2\pi]$, $0 <s$, let
$$\tau_\theta\co  q^{-1}(s) \into q^{-1}(s e^{i\theta})$$
be the parallel transport along $\gamma(t) = s e^{i\theta t}$, $ 0\leq t \leq 1$. 
Take the restriction 
$$\widehat \tau_\theta = \tau_\theta|_{(q^{-1}(s) \setminus \Sigma_s)}\co  q^{-1}(s) \setminus \Sigma_s \into q^{-1}(se^{i\theta}) \setminus \Sigma_{s e^{i\theta}}.$$
For $w \in D^2$, let 
$$\Phi_w \co  q^{-1}(w) \setminus \Sigma_w \into T^*S^n \setminus S^n$$
be the restriction of $\Phi$. Then it turns out
$$\Phi_{se^{i\theta}} \circ \widehat \tau_\theta \circ \Phi_s^{-1}\co  T^*S^n \setminus S^n \into T^*S^n \setminus S^n$$
satisfies 
\begin{gather}
 (\Phi_{se^{i\theta}} \circ \widehat \tau_\theta \circ \Phi_s^{-1})(u,v)  = \phi^{R_s(\mu)}_\theta
= \sigma_{\theta \widetilde R_s'(|v|)}(u,v). \label{transport}
\end{gather}
Here, $\widetilde R_s$ is a certain function (namely $\widetilde R_s(t) = \frac{1}{2}t-\frac{1}{2}(t^2 +s^2/4)^{1/2}$) with the following properties:
$$\lim_{t \rightarrow 0^+}\widetilde R_s'(t) = 1/2,$$
$$\widetilde R_s''(t) <0, \text{ for } t \geq 0,$$
$$\lim_{t \rightarrow \infty} \widetilde R_s'(t) = 0, \text{ and}$$
$$\widetilde R_s(-t) = R_s(t) -t.$$
These properties together with (\ref{transport}) show that  
when $|v| \approx \infty$, we have
$$(\Phi_{se^{i\theta}} \circ \widehat \tau_\theta \circ \Phi_s^{-1})(u,v) \approx (u,v)$$ 
and when $|v| \approx 0$, we have
$$(\Phi_{se^{i\theta}} \circ \widehat \tau_\theta \circ \Phi_s^{-1})(u,v) \approx \sigma_{\theta/2}(u,v).$$ 
In particular, if $n=1$ (so $T^*S^n = T^*S^1  \cong S^1 \times \bbR$) and $\theta = 2\pi$, then we get a classical Dehn twist.

\subsection{Complexifications of standard Morse functions on $\bbC^4$}
\label{standardMorse}
Now we restrict our attention to $\bbC^4$. Let  
$$q_0(z) = z_1^2  + z_2^2 +z_3^2 + z_4^2, $$
$$q_2(z) = z_1^2  + z_2^2 -z_3^2 - z_4^2, $$
$$q_4(z) = -z_1^2  - z_2^2 -z_3^2 -z_4^2. $$
The discussion in the last section applies directly to $q_0$. 
Let us re-label the important objects above:
$$\Phi^0 = \Phi, \rho_s^0 = \rho_s, \Sigma^0_w = \Sigma_w,  k^0 = k, \tau_\gamma^0 =\tau_\gamma, \tau_\theta^0 = \tau_\theta.$$

To understand $q_2$, let
$$\alpha_2\co  \bbC^4 \into \bbC^4, \alpha_2(z) = (z_1, z_2, iz_3,iz_4)$$
Then $\alpha_2$ gives an isomorphism of Lefschetz fibrations from 
$$q_2\co  \bbC^4 \into \bbC \text{ to }q_0\co   \bbC^4 \into \bbC.$$
The basic objects above become:
$$\Phi^2 = \Phi^0 \circ \alpha_2, $$ 
$$\rho_s^2 = \rho_s^0 \circ \alpha_2, $$ 
$$\Sigma^2_w = \alpha_2(\Sigma_w^0), $$ 
$$k^2(z) = k^0(\alpha_2(z)) = |z|^4 -|q_2(z)|^2,$$
$$\tau_\theta^2 = \alpha_2^{-1} \circ \tau_\theta^0 \circ \alpha_2.$$ 

Then everything is as before. In particular, the formula for the transport map $\tau_\theta^2$ is the same as in (\ref{transport}):
$$(\Phi_{se^{i\theta}}^2 \circ \widehat \tau_\theta^2 \circ (\Phi_s^2)^{-1})(u,v) = \sigma_{\theta \widetilde R_s'(|v|)}(u,v).$$
We could treat $q_4$ in a similar way, but instead we will use the fact that
$$q_4^{-1}(-s) = q_0^{-1}(s), \, s>0.$$
Namely, for $q_4$ we have a canonical identification
$$q_4^{-1}(-s) \into T^*S^3$$
(rather than the the usual $q^{-1}(s) \into T^*S^n$) given by 
$$\rho_s^0\co  q_0^{-1}(s) = q_4^{-1}(-s) \into T^*S^3.$$
Thus, for $w \in \bbC$, $s >0$, $z \in \bbC^4$, we take
$$\Phi^4_{w} =\Phi^0_{-w},$$
$$\rho_{-s}^4 = \rho_{s}^0,$$
$$\Sigma^4_w = \Sigma^0_{-w},$$
$$\Sigma^4 = \Sigma^0,$$
$$k^4(z) = k^0(z) = |z|^4 -|q_4(z)|^2.$$
(Note that $\Phi^2$ and $\Phi^4$ are not defined analogously, but using this $\Phi^4$ will save us a little trouble later.) 

\subsection{Cutting down the local models} \label{cutdown}
We describe a standard way to cut down the local models $q_0, q_2, q_4\co  \bbC^4 \into \bbC$ 
so that the fibers become $D(T^*S^3)$ and the transport maps become \emph{equal} to the identity 
near the boundary of the fibers (see lemma 1.10 in \cite{LES} for details). 
\\
\newline Fix $r>0$. Note this should be the \emph{same} $r$ as in \S \ref{4mfds}.
For $i=0,2,4$, let
$$E^{\, i}_{loc} = \{z \in \bbC^4 : k_i(z) \leq 4r^2, |q_i(z)| \leq 1\},$$
$$\pi^{\, i}_{loc} = q_i|_{E_{loc}^{\, i}}: E_{loc}^{\, i} \into D^2.$$
Then $E^{\, i}_{loc}$ is a manifold with corners and $\Phi^{\, i}$ restricts to a fiber preserving diffeomorphism, for which we keep the same notation
$$\Phi^{\, i}\co  E^{\, i}_{loc} \setminus \Sigma^{\, i} \into [D_r(T^*S^3) \setminus S^3] \times D^2.$$
Similarly $\rho^{\, i}$ restricts to canonical map
$$\rho^{\, i}_s\co  (\pi_i^{loc})^{-1}(s) \into D_r(T^*S^3), 0< s \leq 1.$$
We equip $E^i_{loc}$ with an exact symplectic form 
$\omega_i$ such that 
$$\omega_i|_{E_{loc}^i \cap k_i^{-1}([4(r/2)^2, 4r^2])} = \Phi^*(\omega_{T^*S^3}|_{D_{[r/2,r]}(T^*S^3) \times D^2})$$
and 
$$\omega_i|_{E_{loc}^i \cap k^{-1}([0,r/4])} = \omega_{\bbC^4}|_{E_{loc}^i \cap k^{-1}([0,r/4])}.$$
(On the intermediate region, $E_{loc}^i \cap k^{-1}([r/4,r/2])$, $\omega_i$ interpolates using some cutoff function.)
We use the same notation for the transport map 
$$\tau^i_\theta \co  (\pi_{loc}^i)^{-1}(s) \into (\pi_{loc}^i)^{-1}(se^{\sqrt{-1}\theta})$$ 
satisfies 
\begin{gather}
(\Phi_{se^{i\theta}}^i \circ \widehat \tau^i_\theta \circ (\Phi_s^i)^{-1})(u,v)  =
\sigma_{\theta  R_s'(|v|)}(u,v). \notag
\end{gather}
Here, $\widehat \tau^i_\theta$ is the restriction to $(\pi_{loc}^i)^{-1}(s) \setminus \Sigma^i_s$, and $R_s$ is a modification of $\widetilde R_s$ by a cut-off function, which satisfies:
\begin{gather} \label{transportB}
R_s(t) = \widetilde R_s(t), t \in [0,r/4]\\ \notag
R_s'(t) =0, t \in [r/2, r],\\ \notag
R_s''(t) <0, t \in [0,r/2]. \\ \notag
R_s(-t) = R_s(t) - t \text{ for small } |t|  \notag
\end{gather}
Thus $\tau_\theta^i$ is equal to the old transport map on $\bbC^4$ near $\Sigma$ and when we trivialize using $\Phi$, it becomes the identity near the boundary of $D_r(T^*S^3)$. 

\section{Construction of three Lefschetz fibrations} \label{pis}
In this section we review how to construct a Lefschetz fibration
with any prescribed symplectic manifold $M'$ as the fiber, and with a single vanishing sphere consisting of any prescribed exact Lagrangian sphere $L' \subset M'$ (see lemma 1.10 in \cite{LES} for details). We apply this construction for each of our local models $\pi^i_{loc}$, $i=0,2,4$ to produce three Lefschetz fibrations $\pi_i$, $i =0,2,4$. If we think of the target of $\pi: E \into D^2$
as containing three real critical values $c_0 < c_2 < c_4$ corresponding to those of $f$, then
$(E, \pi)$ will be constructed so that $(E_i, \pi_i)$ is equal to the part of $(E, \pi)$ lying over a small disk around $c_i$: 
$$(E|_{D_s(c_i)}, \pi|_{D_s(c_i)}) \cong (E_i, \pi_i).$$ 
The regular fiber of $\pi_i$, say $M_i$,  is a exact symplectomorphic to $M$ (the regular fiber of $\pi$) but $M_i$ is in some cases obtained from $M$ by a key twisting operation, which models the transport map $\pi^{-1}(c_2-s) \into \pi^{-1}(c_2+s)$ along a half circle in the lower half plane.

\subsection{A particular choice for $h$ in the definition of $L_4$} \label{particularh}

Before proceding we specify our choice of the function $h$ in the definition of $L_4$ in \S \ref{4mfds}. Namely we set 
$$h(t) = t/2 - R_{1/4}(t)$$
where $R_{1/4}$ is from (\ref{transportB}). (The choice $s=1/4$ comes from the choice of basepoint
(\ref{basepoint}), which comes later.) 

\subsection{Construction of $\pi_0\co  E_0 \into D^2$} \label{pi0}
In this section we construct a Lefschetz fibration 
$$\pi_0\co  E_0 \into D^2$$ 
such that for any $0< s \leq 1$ there is a canonical exact symplectic identification
$$\rho^0_s\co  \pi_0^{-1}(s) \into  M, 0< s \leq 1$$ 
with the vanishing sphere corresponding  to $L_0$. By construction of $M$, we have a canonical 
exact Weinstein embedding 
$$D(T^*S^3) \into M.$$
For convenience we shrink $r>0$ if necessary so that 
the disk bundle $D_r(T^*S^3)$, with respect to the round metric is contained in $D(T^*S^3) \subset M$.
Then we let
$$\phi_{L_0}\co  D_r^{g_{S^3}}(T^*S^3)  \into M$$
denote the restriction of the above Weinstein embedding.
\\
\newline Now take the trivial fibration $p\co  M \times D^2 \into D^2$ and let 
$$U = \phi_{L_0}(D_{(r/2,r]}(T^*S^n))\times D^2 \subset M \times D^2.$$
Now $E_{loc}^0$ has a corresponding subset 
$$W=  (\Phi^0)^{-1}(D_{(r/2,r]} (T^*S^n)\times D^2).$$
We define the total space $E_0$ by taking the quotient of the disjoint union 
$$ [(M \setminus \phi_{L_0}(D_{r/2}(T^*S^n)) \times D^2]\sqcup E_{loc}^0,$$ 
where we identify  $U$ with  $W$ using 
$$(\phi_{L_0} \times id_{D^2}) \circ \Phi^0|_{W}\co  W \into U.$$ 
See figure  \ref{FigureFiberGlue} for a schematic of the 2 dimensional gluing. 

\begin{figure}
\begin{center}
\includegraphics[width=3in]{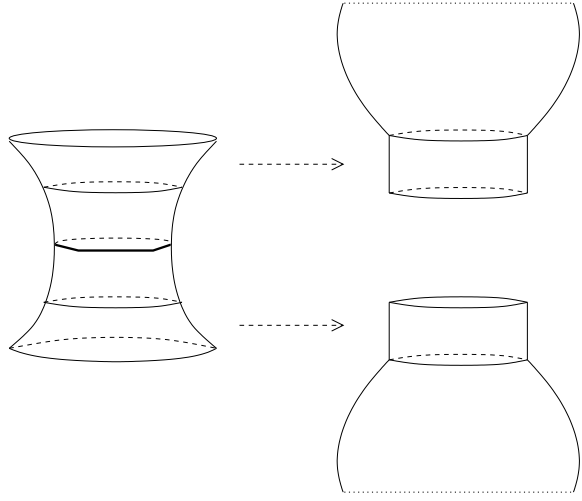} 
\caption{On the left is a 
picture of the fiber of $E_{loc}^0$ (in the case when $E_{loc}^0$ has real dimension 4), where the fiber is symplectomorphic to $D_r(T^*S^1)$; on the right is
$(M \setminus \phi_{L_0}(D_{r/2}(T^*S^1))$. The horizontal arrows indicate the gluing map $W \into U$. }
\label{FigureFiberGlue}
\end{center}
\end{figure} 

Then $\pi_{loc}^0$ and $p$ combine to give a map 
$\pi_0\co  E_0 \into D^2$, which has the structure of an exact symplectic Lefschetz fibration.
\\
\newline For any $0<s\leq 1$ we have the canonical identification 
$$(\rho_{loc}^0)_s\co  \pi_{loc}^{-1}(s) \into D_r(T^*S^n)$$ 
which gives rise to a canonical identification 
$$\rho^0_s\co  \pi_0^{-1}(s) \into M $$
which is defined to be the identity map on $ (M \setminus \phi_{L_0}(D_{r/2}(T^*S^n)) \times \{s\}$, and on $\pi_{loc}^{-1}(s)$ it is defined to be  
$$\phi_{L_0} \circ (\rho_{loc}^0)_s = \phi_{L_0} \circ \Phi^0|_{\pi_{loc}^{-1}(s)}.$$
For  any embedded path 
$$\gamma\co [0,1] \into D^2$$
such that $\gamma(0) = s>0$, $\gamma(1) =0$, the Lefschetz thimble of $\gamma$ is
$$\Delta_\gamma = \cup_t \Sigma_{\gamma(t)}^0 \subset E_{loc}^0 \subset E_0.$$
The vanishing sphere in $\pi^{-1}_0(s)$ is
$$V_\gamma =\Sigma_{s}^0 \subset (\pi_{loc}^0)^{-1}(s) \subset (\pi_0)^{-1}(s).$$
Note that 
$$\rho_s^0(V_\gamma) = L_0 \subset M.$$

\subsection{Construction of $\pi_2\co  E_2 \into D^2$}\label{pi2}
First, we modify $M$ to get a new manifold $M_2$ which will be the regular fiber of $\pi_2$.
For each $j=1, \ldots, k$  take an exact Weinstein embedding for $L_2^{\, j} \subset M$
$$\phi_{L_2^{\, j}} \co  D_r^{g_{S^3}}(T^*S^3)\into M$$
satisfying 
$$\phi_{L_2^{\, j}}|_{D_r(\nu^*K_-)} = \phi_j|_{S^1 \times D^2_r}$$
as in section \ref{4mfds} for some $r$. 
Note this is the same $r$ as in \ref{cutdown}, and \ref{4mfds} (shrink $r$ if necessary). We want to describe a twist operation which will happen near each $L_2^{\, j}$. But it is clearer to describe it near just one $L_2^{\, j}$ to start: Define $T_{\pi/2}^{L_2^{\, j}}(M)$ to be the be the quotient of the disjoint union 
$$[M \setminus \phi_{L_2^{\, j}}(D_{r/2}(T^*S^3)] \sqcup  D_r(T^*S^3),$$ 
where we identify 
$$D_{(r/2, r]}(T^*S^3) \text{ and }\phi_{L_2^{\, j}}( D_{(r/2, r]}(T^*S^3)) \subset M$$
using the map 
$$\phi_{L_2^{\, j}} \circ \sigma_{\pi/2}\co  D_{(r/2, r]}(T^*S_j) \into \phi_{L_2^{\, j}}( D_{(r/2, r]}(T^*L_2^{\, j})).$$
Thus we have glued the neighborhood of $L_2^{\, j}$ back in with a twist by $\sigma_{\pi/2}$.
This makes sense because $\sigma_{\theta}$ maps  $D_{(r/2, r]}(T^*L_2^{\, j})$ diffeomorphically onto itself
for any $\theta \in [0,2\pi]$.
Later we will use: 
$$\sigma_{\pi/2}(D_{(r/2, r]}(\nu^*K_+)) =  D_{(r/2, r]}(\nu^*K_-).$$
Now let 
$$M_2 = T_{\pi/2}(M) = T_{\pi/2}^{L_2^1}T_{\pi/2}^{L_2^2} ... T_{\pi/2}^{L_2^k}(M)$$
be the result of doing this twist operation in a neighborhood of all the $L_2^{\, j}$'s simultaneously. Of course, since the neighborhoods $\phi_{L_2^{\, j}}(D_r(T^*S^3)$ are disjoint these operations are completely independent.
This definition of  $M_2$ is  motivated by lemmas \ref{leftfiber=M} and \ref{halftwist} below. 
\\
\newline Let $S_j$, $j=1, \ldots, k$ denote $k$ copies of $S^3$.
For each $j$ we have a natural exact embedding 
$$\phi_j^2\co  D_r(T^*S^3)\into M_2$$
which is given by the inclusion into the disjoint union  
$$D_r(T^*S_j)\into [M \setminus \sqcup_{j=1}^{\, j =k} \phi_{L_2^{\, j}}(D_{r/2}(T^*S^3))]  \sqcup [\sqcup_{j=1}^{\, j =k} D_r(T^*S_j)],$$ 
followed by the quotient map.
Define 
$$(L_2^{\, j})' = \phi_j^2(S^3) \subset M_2,$$
and take the Weinstein embedding $\phi_{(L_2^{\, j})'}$ for $(L_2^{\, j})'$ to be 
$$\phi_{(L_2^{\, j})'} =\phi_j^2 \co  D_r(T^*S^3)\into M_2.$$
\\
Now take $k$ copies of the local model $\pi_{loc}^2\co  E_{loc}^2 \into D^2$ 
and denote them 
$$(\pi^2_{loc})^{\, j} \co  (E^2_{loc})^{\, j}\into D^2, j=1, \ldots, k.$$
To define 
$$\pi_2\co E_2 \into D^2$$ 
we do a similar quotient construction as for $\pi_0$; the only difference is we do it $k$ times, once for each of the disjoint Weinstein neighborhoods of $(L_2^{\, j})'$, $j=1, \ldots, k$.
That is, near $(L_2^{\, j})'$ we glue in $(E^2_{loc})^{\, j}$, using $\Phi^2$.
Then, there is a canonical identification
$$\rho_s^2\co  \pi_2^{-1}(s) \into M_2, 0< s \leq 1$$
defined as before.
For fixed $0< s \leq 1$, let
$$\gamma_2(t) = s(1-t)$$
so that $\gamma_2(0) =s$, $\gamma_2(1) =0$. 
There are $k$ disjoint Lefschetz thimbles corresponding to this one path $\gamma_2$, one in each 
local model $(E_{loc}^2)^{\, j}$. Namely,
$$\Delta_{\gamma_2}^{\, j} = \cup_t \Sigma_{\gamma_2(t)}^2 = \cup_t \, \alpha_2^{-1}(\Sigma_{\gamma_2(t)}^0) \subset (E_{loc}^2)^{\, j} \subset E_2.$$
The vanishing spheres
$$V_{\gamma_s}^{\, j} = \Sigma_{s}^2 \subset [(\pi_{loc}^2)^{\, j}]^{-1}(s) \subset \pi_2^{-1}(s)$$
satisfy
$$\rho_s^2(V_{\gamma_s}^{\, j}) = (L_2^{\, j})'.$$

\subsection{Construction of $\pi_4\co  E_2 \into D^2$}\label{pi4}
$\pi_4$ will also have regular fiber $M_2$, but at $s =-1$ rather than $s=1$. 
It remains to specify the vanishing sphere. We have already seen $M_2$ has exact Lagrangian spheres $(L_2^{\, j})'$ corresponding to $L_2^{\, j}$. There are also exact Lagrangian spheres
$$L_0', L_4' \subset M_2$$
which correspond to $L_0$, $L_4$. (in the sense of  lemma \ref{halftwist} below).
Namely, define $L_4'$ as a union in $M_2$:
$$L_4' = [L_0 \setminus (\cup_{j=1}^{\, j =k} \phi_j(S^1 \times D^2_{r/2}))] \cup [\cup_{j=1}^{\, j =k}  \phi_{(L_2^{\, j})'}(D_{r}(\nu^*K_+))].$$
Here, $\phi_{(L_2^{\, j})'}$  identifies 
$$D_{[r/2,r]}(\nu^*K_+) \subset D_r(T^*S^3)$$ 
with 
\begin{gather}
 \phi_{L_2^{\, j}}\circ \sigma_{\pi/2}(D_{[r/2,r]}(\nu^*K_+)) = \phi_{L_2^{\, j}}(D_{[r/2,r]}(\nu^*K_-)) = \phi_j(S^1 \times D^2_{[r/2,r]})\subset L_0. \label{twist}
\end{gather}
\\
The definition of $L_0'$ is analogous to that of $L_4$ in \S \ref{4mfds}, \ref{surgery}, as follows.
First set
$$\widetilde L_0' = \phi^{h(\mu)}_{-\pi}(D_r(\nu^*K_-)) \subset D_r(T^*S^3).$$
Note that 
$$\widetilde L_0' \cap D_r(\nu^*K_+) = D_{(r/2,r]}(\nu^*K_+).$$
Then $L_0'$ is defined to be 
$$L_0' = [L_0 \setminus (\cup_{j=1}^{\, j =k} \phi_j(S^1 \times D^2_{r/2}))]
\cup [\cup_{j=1}^{\, j =k} \phi_{(L_2^{\, j})'}(\widetilde L_0') ] \subset M_2.$$
Here, $\phi_{(L_2^{\, j})'}$ identifies 
$$D_{(r/2,r]}(\nu^*K_+)\subset  \widetilde L_0 '$$ 
with 
$$\phi_j(S^1 \times D^2_{[r/2,r]})\subset L_0,$$
as in (\ref{twist}). (See figure \ref{Fiberofpi0} and \ref{Fiberofpi2} below for pictures of the 2 dimensional analogues of $M$, the regular fiber of $\pi_0$, and $M_2$, the regular fiber of $\pi_2$, with their vanishing spheres.)
\begin{figure}
\begin{center}
\includegraphics[width=4in]{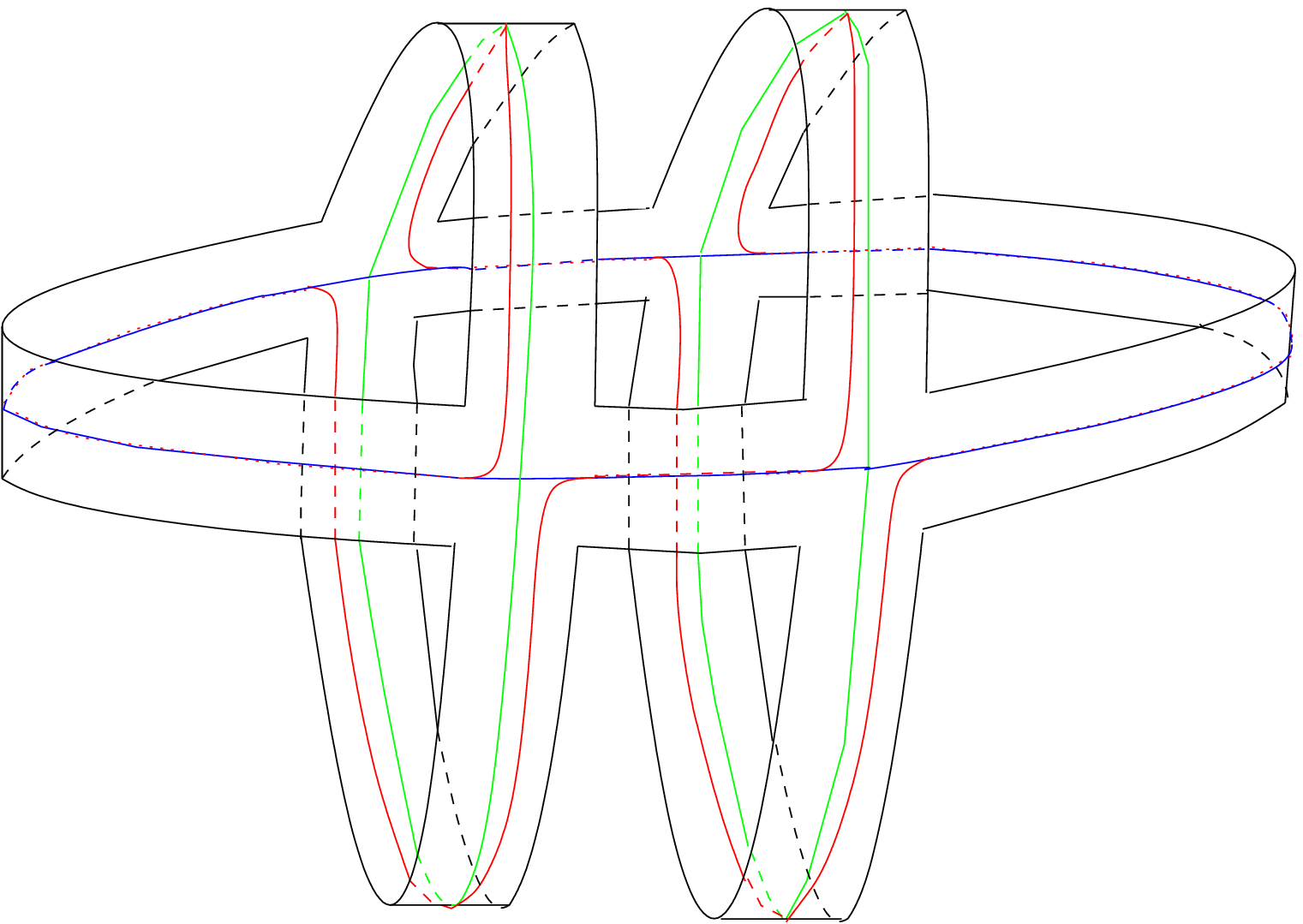} 
\caption{A 2 dimensional version of the regular fiber $M$ of $\pi_0$ at $c_0+1$.
$L_0$ corresponds to the blue curve; $L_2^{\, j}$ $j=1,2$ correspond to the two green curves;
$L_4$  corresponds to the (twisted) red curve.}
\label{Fiberofpi0}
\end{center}
\end{figure} 
\begin{figure}
\begin{center}
\includegraphics[width=4in]{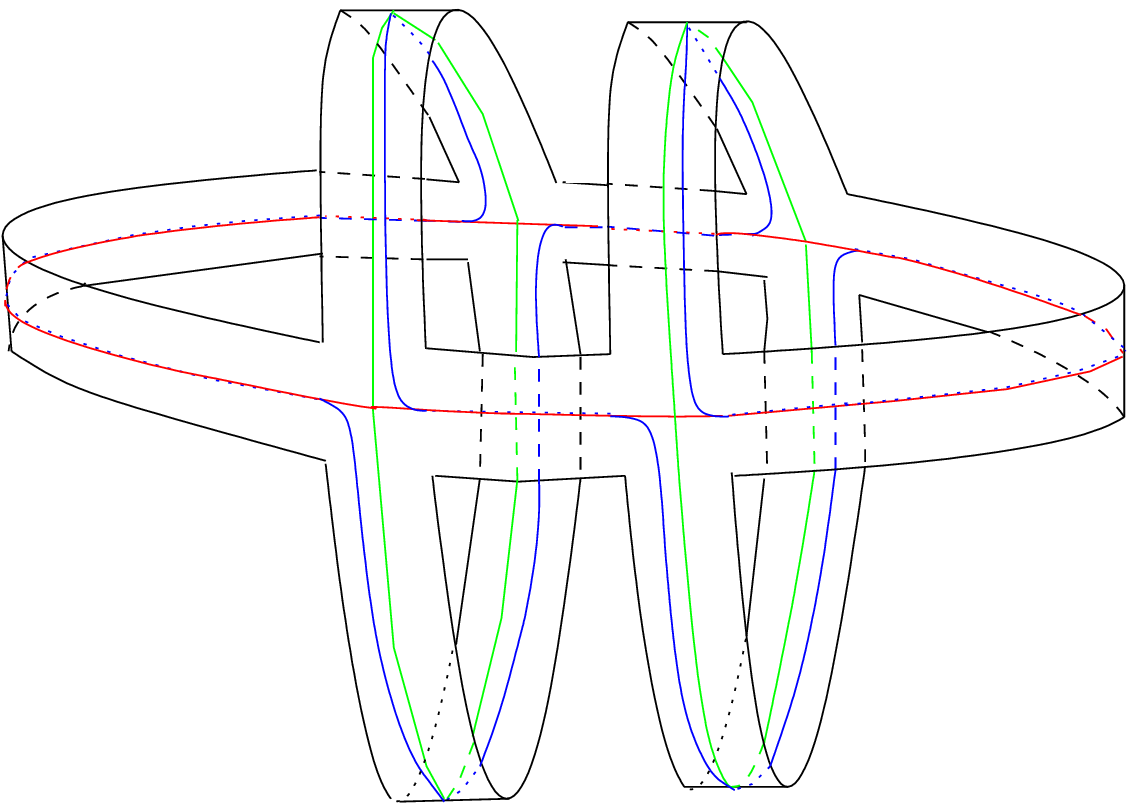} 
\caption{A 2 dimensional version of the regular fiber $M_2$ of $\pi_2$ at $c_2+1$.
$L_0'$ corresponds to the (twisted) red curve; $(L_2^{\, j})'$ $j=1,2$ correspond to the two green curves; $L_4'$  corresponds to the blue curve.}
\label{Fiberofpi2}
\end{center}
\end{figure} 
To construct $\pi_4$, we glue in the local model $\pi_4$ near $(L_4)' \subset M_2$, using $\Phi^4$. For $s>0$ the canonical identification
$$(\rho_{loc}^4)_{-s}\co  (\pi_{loc}^4)^{-1}(-s) \into D_r(T^*S^3)$$
gives rise to a canonical identification
$$\rho_{-s}^4\co  (\pi_{loc}^4)^{-1}(-s) \into M_2$$
at $-s$ rather than $s$.
The Lefschetz thimble for any embedded path $\gamma$, $\gamma(0) = -s$, $\gamma(1) =0$, is given by $$\Delta_\gamma = \cup_t \Sigma_{\gamma(t)}^4 \subset E_{loc}^{4} \subset E_4$$ 
and the vanishing sphere 
$$V_\gamma \subset (\pi_{loc}^4)^{-1}(-s)$$
satisfies 
$$\rho_{-s}^4(V_\gamma) = (L_4)'.$$

\subsection{The transport map for $\pi_i\co  E_i \into D^2$}
We denote the transport map for $\pi_i$ along $\gamma(t) = se^{i\theta t}$, $t \in [0,1]$ by 
$$\tau_\theta^i \co  \pi_i^{-1}(s) \into \pi_i(se^{i\theta}).$$
This is the same notation as for the local models $E_{loc}^i$, but since the transport map for $E_i$ is just the transport map for $E_{loc}^i$ extended by the identity map, this should not cause confusion.

\section{Construction of the Lefschetz fibration $\pi\co  E \into D^2$} \label{pi}
We now have three Lefschetz fibrations $\pi_0,\pi_2,\pi_4$ over the disk $D^2$. It would be convenient to have notation which distinguishes between the different critical points which at the moment are all labeled as $0 \in D^2$. Thus, let $c_0= 0,c_2=2, c_4=4$  and replace $\pi_i$ by $\pi_i + c_i$ but  keep the same notation.
Let $D(c_i)$ denote the disk of radius 1 centered at $c_i$ and let $D_{s}(c_i)$ denote the disk of radius $s$, $0< s \leq 1$. Thus, from now on each $\pi_i$ is a Lefschetz fibration 
$$\pi_i \co  E_i \into D(c_i), \, i=0,2,4,$$
where $c_i$ is the critical value. 
(The choice  $c_i=i$, is of course not essential; 
we just want the different labels $c_0, c_2, c_4$.)

\subsection{The fiber-connect sum of Lefschetz fibrations}\label{fiberconnectsum}

In this section we briefly the definition of the \emph{fiber connect sum} of two  Lefschetz fibrations.
See \cite[p.\,7,\,27]{LES} for details. Let $\pi^1\co  E^1 \into S^1$ and $\pi^2\co  E^2 \into S^2$ be two  Lefschetz fibrations. Here, $S^1$, $S^2$ are two Riemann surfaces with boundary; for us these will both be disks. 
Let $b^1 \in \partial S^1$ and $b^2 \in \partial S^2$. Let $S^1 \# S^2$ denote the boundary connect sum at $b^1$, $b^2$. The main input one needs is an exact symplectic identification of the fibers over $b^1, b^2$:
$$\Psi^{12}\co  (\pi^1)^{-1}(b_1) \into (\pi^2)^{-1}(b_2).$$
Then, there is a Lefschetz fibration 
$$\pi^1 \# \pi^2 \co  E^1 \# E^2 \into S^1 \# S^2$$
called the fiber-connect sum of $\pi_1$ and $\pi_2$, which is of course
obtained by identifying the fibers over $b_1$ and $b_2$. (More formally, one deletes a small neighborhoods of these fibers and identifies along the resulting boundaries.)
As a technical condition, one also needs the symplectic connections (given by the symplectic orthogonals to the fibers) to be \emph{flat} in a neighborhood of $b^1, b^2$. (We will arrange that later for our case.)

\subsection{Construction of $\pi: E \into D^2$ by fiber-connect sum.} Take the boundary connect sum  $D(c_0) \# D(c_2)$ at  $c_0+1$ and $c_2-1$.
Then take $S = [D(c_0) \# D(c_2)] \# D(c_4)$, done at $c_2+1$ and $c_4-1$. Note that $S \cong D^2$.
The plan for this section is to first do coresponding fiber connect sums (as in \S \ref{fiberconnectsum}):
$$\pi_0 \# \pi_2\co  E_0 \# E_2 \into D(c_0) \# D(c_2)$$
and 
$$(\pi_0 \# \pi_2) \# \pi_4\co  (E_0 \# E_2) \# E_4  \into S.$$
Then we will define 
$$\pi = \pi_0 \# \pi_2 \# \pi_4, \, E = E_0 \# E_2 \# E_4.$$

The main step is to specify exact symplectic identifications 
$$\Psi_{02}\co  \pi_0^{-1}(c_0+1) \into \pi_2^{-1}(c_2-1), \text{ and}$$
$$\Psi_{24}\co  \pi_2^{-1}(c_2+1) \into \pi_4^{-1}(c_4-1).$$
First, $\Psi_{24}$ is easy because, for $0<s\leq 1$, there are canonical maps
$$\rho^2_{s} \co \pi_2^{-1}(c_2+s)\into M_2, $$
$$\rho^4_{-s} \co \pi_4^{-1}(c_4-s)\into M_2.$$
So we set 
$$\Psi_{24} =(\rho^4_{-1})^{-1} \circ  \rho^2_{1}.$$ 
\\
The next lemma will tell us how to define $\Psi_{02}$.
\begin{lemma}\label{leftfiber=M} For $0< s \leq 1$, there is a canonical exact symplectic isomorphism  
$$\nu^2_{-s}\co  \pi_2^{-1}(c_2-s)\into M.$$
 \end{lemma} 
\begin{proof} Note that 
$$M_2 \setminus (\cup_{j=1}^{\, j =k} \phi_{(L_2^{\, j})'}(D_{ r/2}(T^*S^3))$$ 
is canonically isomorphic to 
$$M \setminus (\cup_{j=1}^{\, j =k} \phi_{L_2^{\, j}}(D_{ r/2}(T^*S^3)).$$ 
From now on we will identify these. 
Then, by definition, $\pi_2^{-1}(c_2-s)$ is equal to the quotient manifold obtained from the disjoint union:
$$M \setminus (\cup_{j=1}^{\, j =k} \phi_{L_2^{\, j}}(D_{ r/2}(T^*S^3)) \sqcup  [ \sqcup_{j=1}^{\, j =k} [(\pi^2_{loc})^{\, j}]^{-1}(c_2-s)]$$
where for each $j$ we identify 
$$U_j = (\Phi_{-s}^2)^{-1} (D_{(r/2,r]}(T^*S^3)) \subset [(\pi^2_{loc})^{\, j}]^{-1}(c_2-s)$$
with
$$W_j= \phi_{L_2^{\, j}}(D_{[ r/2,r]}^{g_{S^3}}(T^*S^3)) \subset M,$$
using the gluing maps
$$(\phi_{L_2^{\, j}} \circ \sigma_{\pi/2} \circ \Phi_{-s}^2)|_{U_j} \co  U_j \into W_j.$$
Now recall that 
$$\Phi_{-s}^2 = \Phi_{-s}^0 \circ \alpha_2 =  \sigma_{-\pi/2} \circ  \rho_s^0 \circ m(i) \circ \alpha_2.$$
so the $j$th gluing map is
$$(\phi_{L_2^{\, j}} \circ \rho_s^0 \circ m(i) \circ \alpha_2)|U_j \co  U_j \into W_j.$$
\\
Think of $M$ as a quotient manifold in the following tautological way: $M$ is the quotient of 
$$M \setminus (\cup_{j=1}^{\, j =k} \phi_{L_2^{\, j}}(D_{ r/2}(T^*S^3)) \sqcup  [ \sqcup_{j=1}^{\, j =k} D_r(T^*S_j)$$
where $S_j = S^3$ for all $j$ and we glue using 
$$\phi_{L_2^{\, j}}\co  D_{(r/2,r]}(T^*S_j) \into \phi_{L_2^{\, j}}(D_{[r/2,r]}(T^*S^3)).$$
Now define the map 
$$ \nu^2_{-s}\co  \pi_2^{-1}(c_2-s) \into M$$
to be the identity on the common part
$$(M \setminus (\cup_{j=1}^{\, j =k}\phi_{L_2^{\, j}}(D_{ r/2} (T^*S^3)))$$
and on the other parts $[(\pi^2_{loc})^{\, j}]^{-1}(c_2-s)$ define
\begin{gather}
\nu^2_{-s}|_{[(\pi^2_{loc})^{\, j}]^{-1}(c_2-s)}\co  [(\pi^2_{loc})^{\, j}]^{-1}(c_2-s) \into D_r(T^*S_j) \\ \notag  
\nu^2_{-s}|_{[(\pi^2_{loc})^{\, j}]^{-1}(c_2-s)} = \rho_s^0 \circ m(i) \circ \alpha_2 \label{nu}
\end{gather}

Then $\nu^2_{-s}$ is compatible with the gluing maps, so it gives a well defined isomorphism of quotient manifolds.
Namely, the gluing map for $M$ is $\phi_{L_2^{\, j}}$ and the gluing map for 
$\pi_2^{-1}(c_2-s)$ is $\phi_{L_2^{\, j}} \circ \nu^2_{-s}|_{[(\pi^2_{loc})^{\, j}]^{-1}(c_2-s)}.$
\end{proof}

Thus, we define 
$$\Psi_{02} = (\nu^2_{-1})^{-1} \circ \rho^0_1,$$
where 
$$\rho_s^0\co  \pi_0^{-1}(c_0+s) \into M, 0< s \leq 1$$
is the canonical isomorphism.
\\
\newline To do the fiber connect sums, we first need to make each $\pi_i$ flat near $c_i \pm 1$. We accomplish that near the whole boundary of $D(c_i)$ as in remark 1.4 in \cite{LES}. Namely, take a smooth map $\psi \co  D^2 \into D^2$ such that:
\\
\newline $\psi|D_{1/2}$  is just the inclusion $D_{1/2} \into D^2$;
\\$\psi|_{D_{3/4}^2}$ maps $D_{3/4}^2$ diffeomorphically onto $D^2$, by a radial map;
\\$\psi|_{D_{[3/4,1]}^2}$ radially collapses $D_{[3/4,1]}^2$ onto $\partial D^2$.
Here, $D^2_{[a,b]} = \{ x \in D^2 : |x| \in [a,b]\}$.
\\
\newline Then, the pull back $\psi^*\pi_i$ is flat near the boundary of $D(c_i)$, and we replace $\pi_i$ by $\psi^*\pi_i$ but keep the same notation. Note that, since  $\psi|_{D_{1/2}}$  is just the inclusion, $E_i|_{{D_{1/2}(c_i)}}$ has not been modified at all. So, in particular, the transport map
$\tau^i_\theta$ along $\gamma(t) = c_i + se^{i\theta}$, for $0< s \leq 1/2$ is the same as before.
\\
\newline Concretely, when we do the boundary connect sums of the base manifolds $D(c_i)$, the segments $[c_i-1, c_i+1] \subset D(c_i)$ are glued together  to form an interval $I$, and we think of $S$ as embedded in $\bbC$ so that $I \subset \bbR$. In fact, parts of $[c_i-1, c_i+1]$ are chopped off before we glue, so $I$ is shorter than $[c_0-1, c_4+1] = [-1,5]$. Nevertheless we will refer to intervals such as $[c_0-1/4, c_2+1/4]$ with the understanding that this means the corresponding sub-interval of $I$, i.e. $[c_0-1/4,c_2+1] \# [c_2-1, c_2+1/4]$. Also note that, for example, the transport map along 
$[c_0-1/4, c_2+1/4]$,  
$$\tau_{[c_0-1/4, c_2+1/4]} \co  \pi^{-1}(c_0-1/4) \into \pi^{-1}(c_2+1/4)$$
is equal to  the composite 
$$ \tau^2_{[c_2-1, c_2+1/4]} \circ  \Psi_{02} \circ \tau^0_{[c_0-1/4, c_0+1]},$$
where $\tau^2_{[c_2-1, c_2+1/4]}$ and $\tau^0_{[c_0-1/4, c_0+1]}$ are the transport maps for 
$\pi_0$ and $\pi_2$ respectively. 

\section{Computing the vanishing spheres of $\pi\co  E \into D^2$}\label{computingVC}

Fix the  base point $b \in D^2(c_2)$ to be  
\begin{gather} \label{basepoint}
 b = c_2 - 1/4.
\end{gather}
Then $\pi^{-1}(b) = \pi_2^{-1}(c_2-1/4)$ and so, by lemma \ref{leftfiber=M}, there is a canonical isomorphism 
$$\nu^2_{-1/4}\co  \pi^{-1}(b) \into M.$$
In this section we show that for suitable vanishing paths $\gamma_0, \gamma_2, \gamma_4$,
the vanishing spheres 
$$V_{\gamma_1}, V_{\gamma_2}^{\, j}, V_{\gamma_4} \subset \pi^{-1}(b), \, j=1, \ldots, k$$ 
correspond precisely to  $L_0$, $L_2^{\, j}$, $L_4$, under the map $\nu^2_{-1/4}$. Here, $\gamma_2$ will give rise to $k$ disjoint vanishing spheres $V_{\gamma_2}^{\, j} \subset \pi^{-1}(b)$
one for each critical point $x_2^{\, j}$. 
\\
\newline For the critical values  $c_0$ and $c_2$, we take the vanishing paths $\gamma_0, \gamma_2$ in $\bbR$
which parametrize the closed segments $[c_0, b]$ and $[b, c_2]$ at unit speed. 
For $c_4$ we take the composite of two paths $\gamma_4^0$ and $\gamma_4^1$.
Let 
$$\gamma_4^0(t) = c_2 + (1/4)e^{(-\pi + \pi t)i }, t \in [0, 1];$$
this parameterizes the half circle from $c_2-1/4$, $c_2 +1/4$ in the lower-half plane. Let $\gamma_4^1$ be the path in $\bbR$ which parameterizes  the closed segment $[c_2+1/4, c_4]$ 
at unit speed. Then let $\gamma_4$ be the vanishing path from $b$ to $c_4$ which is obtained by following 
$\gamma_4^0$ and then $\gamma_4^1$; see figure \ref{figurevanishingpaths}.
\begin{figure}
\begin{center}
 \includegraphics[width=3in]{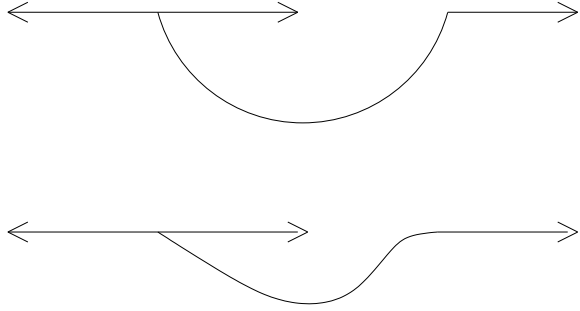} 
\caption{The three vanishing paths $\gamma_0, \gamma_2, \gamma_4$ (top) and
the version where  $\gamma_2$ is smooth (see remark \ref{smoothLefTh})  $\gamma_0, \widetilde \gamma_2, \gamma_4$ (bottom).}
\label{figurevanishingpaths}
\end{center}
\end{figure} 
 Note that $\gamma_4$ is only piece-wise smooth, so it does not give a
smooth Lefschetz thimble (but see remark \ref{smoothLefTh} below). Nevertheless it does give a smooth vanishing sphere 
$$V_{\gamma_4} \subset \pi^{-1}(b)$$
defined as follows.   Take the vanishing sphere
$$V_{\gamma_4^1} \subset \pi^{-1}(c_2+1/4)$$
and then define 
$$V_{\gamma_4} = \tau_{\gamma_4^0}^{-1}(V_{\gamma_4^1}) \subset \pi^{-1}(b),$$
where $\tau_{\gamma}$ denotes the transport along a path $\gamma$.

\begin{remark} \label{smoothLefTh} To apply the results of \cite{S08} one needs \emph{smooth} Lefschetz thimbles (since these are key objects in the Fukaya category associated to a Lefschetz fibration). To bridge this gap, wiggle  $\gamma_4^0$ slightly  to $\widetilde \gamma_4^0$ so that $\widetilde \gamma_4^0$ is tangent to all orders to the real line at $c_2+1/4$ (and with tangent vector pointing to the right; see figure \ref{figurevanishingpaths}). Precisely speaking, $\gamma_4^0$ is smoothly isotopic to $\widetilde \gamma_4^0$. Then, the concatenation of  $\widetilde \gamma_4^0$ and $\gamma_4^1$, say $\widetilde \gamma_4$, is smooth, provided we parametrize both paths so they have  unit speed. 
We express $V_{\widetilde \gamma_4}$ in away similar to $V_{\gamma_4}$ above: 
$$V_{\widetilde \gamma_4} = \tau_{\widetilde \gamma_4^0}^{-1}(V_{\gamma_4^1}).$$
Since $\widetilde \gamma_4^0$ is smoothly isotopic to $\gamma_4^0$  it follows that
$\tau_{\widetilde \gamma_4^0}^{-1}$ is smoothly isotopic to $\tau_{\gamma_4^0}^{-1}$ 
through exact symplectomorphisms
and hence $V_{\widetilde \gamma_4}$ is exact isotopic to $V_{\gamma_4}$. 
Moreover, by making $\widetilde \gamma_4^0$ $C^0$-close to $\gamma_4^0$, we can arrange that 
$V_{\widetilde \gamma_4}$ is $C^0$-close to  $V_{\gamma_4}$, for any desired closeness.
($C^0$ close is best possible since $\widetilde \gamma_2^0$ and $\gamma_2^0$ neccesarily have orthogonal derivatives at the right end-point.) With this understood, 
we stick with our definition of $V_{\gamma_4}$ for convenience.
\end{remark}

We point out that $\tau_{\gamma_2^0}^{-1}$ coincides with the transport map
$$\tau_{-\pi}^2\co  \pi^{-1}_2(c_2-1/4) \into \pi^{-1}_2(c_2+1/4)$$ 
along $\gamma(t) = c_2 + \frac{1}{4}e^{-\pi i t},\,  t \in [0,1]$.
The next lemma describes this map.

\begin{lemma}\label{halftwist} 
\begin{enumerate}
 \item  There is a canonical exact symplectic isomorphism 
$$\tau\co  M_2 \into M$$ \label{tau}
such that 
\item $\tau(L_0') = L_0$,  $\tau((L_2^{\, j})') = L_2^{\, j}$ and $\tau(L_4') = L_4$. \label{tau(L')}
\item \label{tauequalstwist}Under the canonical  identifications $\pi_2^{-1}(c_2+1/4) \cong M_2$,  $\pi_2^{-1}(c_2-1/4)\cong M$, 
$\tau$ becomes  the transport map 
$$\tau_{-\pi}^2 \co \pi_2^{-1}(c_2+1/4) \into  \pi_2^{-1}(c_2-1/4),$$
that is,
$$\nu^2_{-1/4} \circ  \tau_{\pi/2}^2 \circ (\rho_{1/4}^{2})^{-1}= \tau.$$

\end{enumerate}

\end{lemma}

\begin{proof} To prove (\ref{tau}), recall that $M_2$ is defined to be the quotient of 
$$[M \setminus (\cup_j \phi_{L_2^{\, j}}(D_{r/2}(T^*S^3))] \sqcup [\cup_j D_r(T^*S_j)]$$
where $S_j = S^3$ for all $j$ and we have gluing maps 
$$\phi_{L_2^{\, j}} \circ \sigma_{\pi/2}\co  D_{[r/2,r]}(T^*S_j) \into  \phi_{L_2^{\, j}}(D_{[r/2,r]}(T^*S^3)).$$
We can think of $M$ as a quotient of the same space where we use the gluing map $\phi_{L_2^{\, j}}$
instead of $\phi_{L_2^{\, j}} \circ \sigma_{\pi/2}.$
\\
\newline Define 
$$\tau\co  M_2 \into M$$ 
to be the identity on the common part
$$M\setminus \phi_{L_2^{\, j}}(D_r/2(T^*S^3))$$ and define it to be 
$\phi^{h(\mu)}_{\pi}$ on $D_r(T^*S_j)$ for each $j$.
(Here $h$ is from \S \ref{particularh}.) Since 
$$\phi^{h(\mu)}_{\pi}\co  D_r(T^*S^3)\into D_r(T^*S^3)$$
is an extension of 
$$\sigma_{\pi/2}\co  D_{(r/2,r]}(T^*S^3)\into  D_{(r/2,r]}(T^*S^3),$$
it follows that $\tau$ is compatible with the gluing maps and gives a well defined isomorphism.
\\
\newline To prove (\ref{tau(L')}), first note that $\tau((L_2^{\, j})') = L_2^{\, j}$ is obvious, since $(L_2^{\, j})'$ and $L_2^{\, j}$ both correspond to $S_j \subset D_r(T^*S_j)$, and
$\phi_{\pi}^h(\mu)(S_j)=S_j$. For the rest, recall from \S \ref{4mfds}, \ref{surgeryh}, \ref{particularh} that the main ingredient in $L_4$ is
$$\widetilde T = \phi_{\pi}^{h(\mu)}(D_r(\nu^*K_+)) \subset D_r(T^*S_j).$$
Indeed, $L_4$ is defined to be the (overlapping) union 
$$L_4 = L_0 \setminus (\cup_j \phi_j(S^1 \times D^2_{r/2})) \cup (\cup_j \phi_{L_2^{\, j}}(\widetilde T).$$
We can think of $L_0$ as being defined by the same formula, but with $\widetilde T$ replaced by
$$\widetilde T_0 = D_r(\nu^*K_-).$$
Next recall that the main ingredient in $L_4' \subset M_2$ is
$$\widetilde T\,' = D_r(\nu^*K_+) \subset D_r(T^*S_j),$$
and $L_4'$ is defined as 
$$L_4' = L_0 \setminus (\cup_j \phi_j(S^1 \times D^2_{r/2})) \cup (\cup_j \phi_{(L_2^{\, j})'}(\widetilde T\,')),$$
where $\phi_{(L_2^{\, j})'}$ identifies $D_{[r/2,r]}(\nu^*K_+)$ with $\phi_j(S^1 \times D^2_{[r/2,r]}) \subset L_0$.
And similarly the main ingredient in $L_0' \subset M_2$ is
$$\widetilde T_0' = \phi^{h(\mu)}_{-\pi}(D_r(\nu^*K_+)).$$
Now, with all this in mind, we have $\tau(L_0') = L_0$  
because in each $D_r(T^*S_j) = D_r(T^*S^3)$, we have
$$\tau(\widetilde T_0') = \phi^{h(\mu)}_{\pi}( \phi^{h(\mu)}_{-\pi}(D_r(\nu^*K_+)) = D_r(\nu^*K_+) =\widetilde T_0.$$
And  $\tau(L_4') = L_4$ because
$$\phi^{h(\mu)}_{\pi}(\widetilde T\,') = \phi^{h(\mu)}_{\pi}(D_r(\nu^*K_+)) = \widetilde T.$$

To prove (\ref{tauequalstwist}), consider
$$\tau_{-\pi}^2\co  \pi_2^{-1}(c_2+1/4) \into \pi_2^{-1}(c_2-1/4)$$
and focus on its action near one of the $k$ local models $(E_{loc}^2)^{\, j}$ 
(where it is supported). Thus we consider the restrictions
$$(\tau^2_{-\pi})^{\, j} = (\tau^2_{-\pi})|_{[(\pi^2_{loc})^{\, j}]^{-1}(c_2+1/4)}\co  [(\pi^2_{loc})^{\, j}]^{-1}(c_2+1/4) \into [(\pi^2_{loc})^{\, j}]^{-1}(c_2-1/4).$$
Now (\ref{transportB}) and (\ref{k(H)}) imply that
\begin{gather*}
(\Phi^2_{-1/4} \circ (\widehat\tau^2_{-\pi})^{\, j} \circ (\Phi^2_{1/4})^{-1})(u,v) = \sigma_{-\pi R_{1/4}'(|v|)}(u,v) = \phi_{-\pi}^{R_{1/4}(\mu)}(u,v).
\end{gather*}
(Here $(\widehat\tau^2_{-\pi})^{\, j}$ denotes the restriction away from the vanishing sphere.)
Expanding out the definitions of $\Phi^2_{\pm 1/4}$, and using 
$\tau_{-\pi}^2 = \alpha_2 \circ \tau_{-\pi}^0  \circ \alpha_2^{-1}$, we get
\begin{gather*}
(\sigma_{-\pi/2} \circ \rho_{1/4}^0 \circ m(i) \circ \alpha_2)  \circ \widehat \tau_{-\pi}^2 \circ (\rho_{1/4}^0 \circ \alpha_2)^{-1} \\
=\sigma_{-\pi/2} \circ \rho_{1/4}^0 \circ m(i) \circ \widehat \tau_{-\pi}^0 \circ (\rho_{1/4}^0)^{-1} = \phi_{-\pi}^{R_{1/4}(\mu)}.
\end{gather*}
Note that $\phi_{-\pi}^{R_{1/4}(\mu)}$ does not extend continuously over the zero section, but
composing on the left with $\sigma_{\pi/2}$ yields a map which does extend over the section, namely
$\phi^{h(\mu)}_{\pi} = \tau|_{D_r(T^*S_j)}$. Indeed:
\begin{gather}  \notag
\rho_{1/4}^0\circ m(i) \circ \tau_{\pi}^0 \circ (\rho_{1/4}^0)^{-1} \\ \label{fuzzynuts}
= \sigma_{\pi/2} \circ \phi_{-\pi}^{R_{1/4}(\mu)}\\ \notag
= \phi^{\mu/2}_{\pi} \circ \phi_{\pi}^{ - R_{1/4}(\mu)}  \\ \notag
 = \phi^{h(\mu)}_{\pi}
= \tau|_{\phi_{L_2^{\, j}}(D_r(T^*S_j))}, \notag
\end{gather}
where we recall $h(t) = t/2 - R_{1/4}(t)$. 
\\
\newline Now conjugate
$$(\tau^2_{-\pi})^{\, j} \co  [(\pi^2_{loc})^{\, j}]^{-1}(1/4) \into [(\pi^2_{loc})^{\, j}]^{-1}(-1/4)$$
to get
\begin{gather*}  
\nu^2_{-1/4} \circ  (\tau_{-\pi}^2)^{\, j} \circ (\rho^2_{1/4})^{-1} \co  D_r(T^*S^3) \into D_r(T^*S^3), 
 \end{gather*}
which is the restriction of 
$$\nu^2_{-1/4} \circ  \tau_{-\pi}^2 \circ (\rho_{1/4}^2)^{-1}\co  M_2 \into M $$ 
to $D_r(T^*S_j)\subset M_2$. By definition of $\nu^2_{-1/4}$ and $\rho^2_{1/4}$, and using (\ref{fuzzynuts}) for the last step, we have 
\begin{gather*}  
\nu^2_{-1/4}\circ  (\tau_{-\pi}^{2})^{\, j}  \circ (\rho_{1/4}^2)^{-1}\\
=  (\rho_{1/4}^0 \circ m(i) \circ \alpha_2)  \circ \tau_{-\pi}^2 \circ (\rho_{1/4}^0  \circ \alpha_2)^{-1}\\
= \rho_{1/4}^0 \circ m(i)  \circ \tau_{-\pi}^0 \circ (\rho_{1/4}^0)^{-1} 
= \tau|_{D_r(T^*S_j)}
\end{gather*}
This shows $\nu^2_{-1/4} \circ  \tau_{-\pi}^2 \circ (\rho^2_{1/4})^{-1}$ and $\tau$ agree on each neighborhood $D_r(T^*S_j)\subset M_2$. Since both maps equal the identity outside of
$\cup_{j=1}^{\, j =k} \, D_r(T^*S_j)$, this shows 
$$\nu^2_{-1/4} \circ  \tau_{-\pi}^2 \circ (\rho^2_{1/4})^{-1} = \tau\co  M_2 \into M.$$ \end{proof}
\begin{remark} Just as a sanity check, let's look at the other map $\tau^2_{\pi}$ (as opposed to  $\tau^2_{-\pi}$). 
In the last lemma we saw that
\begin{gather*}
 (\nu^2_{-1/4} \circ \tau_{-\pi}^{2} \circ (\rho^2_{1/4})^{-1})|_{(E_{loc}^2)^{\, j}}\\
= \sigma_{\pi/2} \circ \phi_{-\pi}^{R_{1/4}(\mu)}.
\end{gather*}
And the same calculation shows that
\begin{gather*}
 (\nu^2_{-1/4}\circ \tau_{\pi}^{2} \circ (\rho^2_{1/4})^{-1})|_{(E_{loc}^2)^{\, j}}\\
= \sigma_{\pi/2}\circ \phi_{\pi}^{R_{1/4}(\mu)}.
\end{gather*}
Then the total monodromy $\tau^2_{2\pi} = (\tau^2_{-\pi})^{-1} \circ \tau^2_\pi $
coresponds to 
$$ \phi_{\pi}^{ R_{1/4}(\mu)} \circ \sigma_{\pi/2}^{-1} \circ
\sigma_{\pi/2} \circ \phi_{\pi}^{ R_{1/4}(\mu)} = \phi_{2\pi}^{R_{1/4}(\mu)}$$ 
as expected.
\end{remark}

We now describe the vanishing spheres corresponding to $\gamma_0, \gamma_2, \gamma_4$.

\begin{lemma}\label{vanishingcycles} Under the canonical isomorphism 
$$\nu^2_{-1/4}\co  \pi^{-1}(b) \into M$$
the vanishing spheres $V_{\gamma_0}$,  $V_{\gamma_2}^{\, j}$, $V_{\gamma_4}$ correspond respectively to
$$L_0, L_2^{\, j}, L_4 \subset M.$$
 \end{lemma}

\begin{proof} Let $V_0 \subset \pi_0^{-1}(c_0+1/4)$ be the vanishing sphere corrresponding to $[c_0, c_0+1/4]$. 
In \S \ref{pi0} we noted that 
$$\rho^0_{1/4}(V_0) = L_0.$$
The vanishing sphere $V_{\gamma_0}$ is 
$$V_{\gamma_0} = \tau_{[c_0 + 1/4, c_2-1/4]}(V_0).$$
At the end of \S \ref{pi} we mentioned that the transport map along $[c_0 + 1/4, c_2-1/4]$ is
$$\tau_{[c_0 + 1/4, c_2-1/4]} =  \tau^2_{[c_2 - 1, c_2-1/4]} \circ \Psi_{02}\circ  \tau^0_{[c_0 + 1/4, c_0+1]}.$$
Now $\tau^0_{[c_0 + 1/4, c_0+1]}$ does not really affect $V_0$ in the sense that 
$\tau^0_{[c_0 + 1/4, c_0+1]}(V_0)$ satisfies 
$$\rho^0_{1}(\tau^0_{[c_0 + 1/4, c_0+1]}(V_0)) = L_0.$$
Set 
$$V_2 = (\Psi_{02}\circ  \tau^0_{[c_0 + 1/4, c_0+1]})(V_0).$$
Then, since $\Psi_{02} =  (\nu^2_{-1})^{-1} \circ \rho^0_{1},$ we have
$$\nu^2_{-1}(V_2) =  \rho^0_{1}(\tau^0_{[c_0 + 1/4, c_0+1]}(V_0)) = L_0 \subset M.$$ 
Now, again, $\tau^2_{[c_2 - 1, c_2-1/4]}$ does not really affect $V_2$ in the sense that
$$\nu^2_{-1/4}(\tau^2_{[c_2 - 1, c_2-1/4]}(V_2)) = L_0.$$
But $\tau^2_{[c_2 - 1, c_2-1/4]}(V_2) = V_{\gamma_0}$, so that proves $\nu^2_{-1/4}(V_{\gamma_0}) = V_{\gamma_0}$.
\\
\newline In \S \ref{pi2} we saw that $V_{\gamma_2}^{\, j} \subset \pi_2^{-1}(c_2-1/4)$
satisfies
$$V_{\gamma_2}^{\, j} = \Sigma^2_{-1/4}= \alpha_2(\sqrt{-1/4}S^3) \subset [(\pi^2_{loc})^{\, j}]^{-1}(-1/4).$$
By (\ref{nu}),
$$\nu^2_{-1/4}|_{[(\pi^2_{loc})^{\, j}]^{-1}(-1/4)} = \rho_{1/4}^0 \circ m(i) \circ \alpha_2,$$
and so we have
$$\nu^2_{-1/4}(V_{\gamma_2}^{\, j}) = \rho_{1/4}^0(\sqrt{1/4}S^3) = L_2^{\, j} \subset M,$$
because $\sqrt{1/4}S^3 = \Sigma^0_{1/4} \subset (\pi^0_{loc})^{-1}(1/4)$.
\\
\newline To analyze $V_{\gamma_4}$, first consider
the vanishing sphere 
$$V_4 \subset \pi^{-1}(c_4-1/4) = \pi_4^{-1}(c_4-1/4)$$
corresponding to the path $[c_4-1/4, c_4]$. In \S \ref{pi4} we saw that
$$\rho_{-1/4}^4(V_4) = (L_4)' \subset M_2.$$
Now 
$$V_{\gamma_4^1} = \tau_{[c_2+1/4, c_4-1/4]}^{-1}(V_4) \subset \pi_2^{-1}(c_4+1/4).$$ 
Using 
$$\tau_{[c_2+1/4, c_4-1/4]}^{-1} = (\tau^2_{[c_2+1/4, c_2+1]})^{-1}\circ \Psi_{24}^{-1} \circ 
(\tau^4_{[c_4-1, c_4-1/4]})^{-1}$$ 
and arguing as we did for $V_{\gamma_0}$ one sees that
$V_{\gamma_4^1}$ satisfies
$$\rho^2_{1/4}(V_{\gamma_4^1}) = L_4' \subset M_2.$$
Then 
$$V_{\gamma_4} = \tau_{\gamma_4^0}^{-1}(V_{\gamma_4^1}) = \tau^2_{-\pi}(V_{\gamma_4^1})$$
Therefore, using  lemma \ref{halftwist}, we get
$$\nu^2_{-1/4}(V_{\gamma_4}) = (\nu^2_{-1/4} \circ \tau^2_{-\pi} \circ (\rho^2_{1/4})^{-1})(L_4') = \tau(L_4') = L_4 \subset M.$$
\end{proof}

\section{Construction of $N \subset E$}\label{ConstructionN}
In this section we construct an exact Lagrangian embedding $N \subset E$ and prove Theorem $A$ (see Theorem \ref{main}). We also discuss some ways of refining Theorem $A$ in remark \ref{refinements}, and we give a detailed sketch of the proof that $E$ is homotopy equivalent to $N$ in Proposition \ref{Prophomotopyequiv}.
\\
\newline 
We first construct a Lagrangian submanifold $\widetilde N \subset E$ and then check that $\widetilde N$ is diffeomorphic to $N$. $\widetilde N \subset E$ will be defined  as the union of several Lagrangian  manifolds, say $N_i$, with boundary (sometimes with corners). This decomposition of $\widetilde N$ is essentially the same as the handle-type decomposition which appears in \cite[pages 27-32]{Milnor}. The fact that the union of the $N_i$ is diffeomorphic to $N$ is essentially Theorem 3.13 there. 
\\
\newline Let
$$N_0 = \Delta_{[c_0, c_2-1/10]},$$
$$N_4 = \Delta_{[c_2+1/10, c_4]}.$$
These are the Lefschetz thimbles over the indicated intervals. They correspond to the $0-$ and $4-$ handles of $N$. Let
$$f_2= q_2|_{E_{loc}^2 \cap \bbR^4} \co  E_{loc}^2 \cap \bbR^4 \into \bbR,$$
and let 
$$N_2^{\, j} = ((E_{loc}^2)^{\, j} \cap \bbR^4)  \cap f_2^{-1}([c_2-1/2, c_2+1/2]).$$
This corresponds to the $j$th $2-$handle of $N$. 
\\
\newline Recall from \S \ref{pi2} that $E_2$ is a certain quotient  of the disjoint union 
$$[M_2 \setminus (\cup_{j=1}^{\, j =k}\phi_{(L_2^{\, j})'}(D_{r/2}(T^*S^3)))] \times D^2(c_2) \sqcup \big( \sqcup_{j=1}^{\, j =k} (E_{loc}^2)^{\, j}) \big).$$
Recall $L_4' \subset M_2$ is defined as
$$L_4' = [L_0 \setminus(\cup_{j=1}^{\, j =k} \phi_j(S^1 \times D^2_{r/2}) )] \cup \phi_{(L_2^{\, j})'}(D_{r}(\nu^*K_+)),$$
where $\phi_{(L_2^{\, j})'}$ identifies $D_{(r/2,r]}(\nu^*K_+)$ with  $\phi_j(S^1 \times D^2_{(r/2,r]})$.
Define $N_2^{triv} \subset E_2$ by 
\begin{eqnarray*}
 N_2^{triv} =[L_0 \setminus(\cup_{j=1}^{\, j =k} \phi_j(S^1 \times D^2_{r/2}) )] \times [c_2-1/2, c_2 +1/2] \\
 \subset [M_2 \setminus (\cup_{j=1}^{\, j =k}\phi_{(L_2^{\, j})'}(D_{r/2}(T^*S^3)))] \times D^2_{1/2}(c_2). 
\end{eqnarray*}
Then, $\widetilde N \subset E$ is defined to be the union
$$\widetilde N = N_0 \cup (\cup_j \,N_2^{\, j}) \cup N_2^{triv} \cup N_4.$$ 
See figure \ref{handledec} in \S \ref{intro} for the 2-dimensional version of this. (Note that the pieces over-lap.)

\begin{thm}\label{main}
\begin{enumerate}
\item \label{tildeN} $\widetilde N \subset E$ is a smooth closed exact Lagrangian submanifold. 
\item There is a diffeomorphism $\alpha \co  N \into \widetilde  N$.\label{alpha}
\item $\pi(\widetilde N) = [c_0, c_4]$. \label{[c_0,c_4]}
\item All critical points of $\pi$ lie on $\widetilde N$, and in fact $Crit(\pi) = \alpha(Crit(f))$. \label{crit}
\item There is a diffeomorphism $\beta\co  \bbR \into \bbR$ such that 
$$\beta \circ  \pi|_{\widetilde N}  \circ \alpha = f\co  N \into \bbR.$$ \label{pi=f}
\end{enumerate}
\end{thm}
\begin{proof} First we prove that $\widetilde N$ is a smooth manifold diffeomorphic to $N$. Define
$$N_2 = (\cup_{j=1}^{\, j =k} N_2^{\, j}) \cup N_2^{triv}.$$
The gluing map in the definition of $E_2$ is the composite of   
$$\Phi^2\co  (E_{loc}^2)^{\, j} \cap (k^2)^{-1}([4(r/2)^2,4r^2]) \into D_{[r/2,r]}(T^*S^3)\times D^2(c_2)$$
and
$$\phi_{(L_2^{\, j})'} \co  D_{[r/2,r]}(T^*S^3) \into M_2.$$
A direct calculation shows that $\Phi^2$ 
satisfies 
$$(\Phi^2)( N_2^{\, j} \cap (k^2)^{-1}([4(r/2)^2,4r^2]) = D_{[r/2,r]}(\nu^*K_+) \times [c_2-1/2,c_2+1/2].$$
Indeed, if $z = (x_1, x_2, x_3, x_4)+i(0,0,0,0) \in \bbR^4 \subset \bbC^4$ then 
$$\Phi_2(z) = \Phi((x_1, x_2,0,0)+(0,0, ix_3, ix_4)).$$
Let 
$$s = q_2((x_1, x_2,0,0)+(0,0, ix_3, ix_4)) \in [-1/4,1/4].$$
If $s < 0$, then 
\begin{gather*}
\Phi_2(z) = (\sigma_{-\pi/2} \circ \rho_s \circ m(i))((x_1, x_2,0,0)+(0,0, ix_3, ix_4)) \\
= \sigma_{-\pi/2} \circ \rho_s((0,0, -x_3, -x_4)+(ix_1, ix_2,0,0)) \\
= \sigma_{-\pi/2}((0,u),(v,0)) \text{ for some } u,v \in \bbR^2\\
= ((u',0),(0,v') \text{ for some } u',v' \in \bbR^2 
\end{gather*}
Similarly, if $s \geq 0$ then $\Phi_2(z) = \rho_s(\alpha_2(z))$ and we get the same conclusion.
\\
\newline 
We recalled earlier in this section that $\phi_{(L_2^{\, j})'}$ identifies 
$D_{[r/2,r]}(\nu^*K_+)$ with 
$$\phi_j(S^1 \times D_{(r/2, r]}^2) \subset L_0.$$
This shows that in the quotient space $E_2$  
$$N_2^{\, j} \cap (k^2)^{-1}((4(r/2)^2,4r^2]) \subset N_2^{\, j}$$ is identified with 
$$\phi_j(S^1 \times D_{(r/2, r]}^2) \times [c_2-1/2, c_2+1/2] \subset N_2^{triv}.$$
Therefore the union $N_2$ is smooth. In  fact,  
$N_2$ is diffeomorphic to $f^{-1}([1,3]) \subset N$, with $\pi|_{N_2}$ equivalent to $f|_{f^{-1}([1,3])}$. 
This essentially follows from the fact that we are using the correct framing $\phi_j$.
(For more details the reader can consult Milnor, \cite[pages 27-32]{Milnor} and especially Theorem 3.13.
See remark \ref{Milnorcomparison} below where we point out some small differences between what we have done here 
and Milnor's set up.)
\\
\newline Next recall that $E_2$ is fiber connect-summed to $E_4$ using the natural identifications
$$\rho_s^2\co  \pi_2^{-1}(s) \into M_2, s >0$$
and 
$$\rho_{-s}^4\co  \pi_4^{-1}(-s) \into M_2, -s<0.$$
For any $s>0$ we have
$$\rho_{-s}^4(N_4 \cap  \pi_4^{-1}(-s) ) = L_4'$$
and 
$$\rho_{s}^2(N_2 \cap  \pi_2^{-1}(s)) = L_4'.$$
Therefore 
$N_2$ and $N_4$ glue together smoothly over the interval $[c_2-1/10, c_2+1/2]$
to form a manifold diffeomorphic to $f^{-1}([2-s,4])$, where $s>0$ is small.
\\
\newline Similarly, recall that $E_2$ is fiber connect-summed to $E_4$ using the natural identifications
$$\nu^2_{-s} \co  \pi_2^{-1}(-s) \into M$$
and 
$$\rho^0_s \co  \pi_0^{-q}(s) \into M.$$
For any $s>0$ we have 
$$\rho_{s}^0(N_0 \cap  \pi_0^{-1}(s) ) = L_0$$
and 
$$\nu_{-s}^2(N_2 \cap  \pi_2^{-1}(-s)) = L_0.$$
Therefore 
$N_2 \cup N_4$ and $N_0$ glue together smoothly over the interval $[c_2-1/10, c_2+1/2]$
to form a manifold diffeomorphic to $N$.
\\
\newline $N_2$ is exact Lagrangian since $N_2^{triv}$ and $N_2^{\, j}$ are, and each overlap region is connected.
Similarly $N_0, N_4$ are exact Lagrangian and overlap with $N_2$ in connected regions, hence $\widetilde N$ is exact Lagrangian. 
\\
\newline The fact that $\pi(N) = [c_0, c_4]$ and the critical points of $\pi$ correspond to those of $f$ is obvious by construction of $N$. The fact that $\pi|_{N}$ is equivalent to $f\co  N \into \bbR$
follows by comparing $\pi$ to $f$ on each handle of the respective handle-type decompositions
(the handle decomposition of $f$ being the one in \cite{Milnor});
on the over-laps between the handles both $\pi$ and $f$ are both just projection to the interval.
\end{proof}
\begin{remark}\label{Milnorcomparison}
Milnor's approach in  \cite[pages 27-32]{Milnor} can be summarized as follows. Milnor identifies $f_2^{-1}([-1,1]) \cap (k^2)^{-1}([4(r/2)^2,4r^2])$  with $S^1 \times D^2_{[r/2,r]}$ in two steps: First he identifies 
$$S^1 \times D^2_{[\varphi(r/2),\varphi(r)]} \, \text{ with } \, f_2^{-1}(-1) \cap (k^2)^{-1}([4(r/2)^2,4r^2])$$  
using a map $\eta$ which involves $\sinh$ and $\cosh$ (Here $\varphi\co  [0,\infty) \into [0,\infty)$ is a certain diffeomorphism whose formula is not important, but see the end of this remark for that.) Second, he uses gradient flow to go the other fibers. 
\\
\newline We, on the other hand, identify 
$$f_2^{-1}([c_2-1/2,c_2+ 1/2]) \cap (k^2)^{-1}([4(r/2)^2,4r^2]) \,\text{ with } \, D_{[r/2,r]}(\nu^*K_+) \cong S^1 \times D^2_{[r/2,r]}$$ 
in just one step using $\Phi^2$. 
To compare our approach to Milnor's we express $\Phi^2$ in two steps as follows. Recall that 
lemma \ref{radial} says $\Phi^2 = \rho_0 \circ \widetilde \Phi^2$, where $\widetilde \Phi_2$ is radial symplectic flow 
to $\pi_2^{-1}(0) \setminus \{0\}$. This implies that $\Phi^2$ can be expressed in two steps as symplectic flow to $\pi_2^{-1}(-1/2)$, followed by $\Phi_{-1/2}^2$. Lemma \ref{realtransport} below shows that symplectic flow along the real part and gradient flow agree up to reparameterization, so that implies that the first step of Milnor's approach agrees with our first step (up to isotopy). As for comparing $\Phi_{-1/2}$ and $\eta$, recall that 
for each $x \in f_2^{-1}(c_2-1/2) \cap (k^2)^{-1}([4(r/2)^2,4r^2])$, we have 
$$\Phi_2(x) = (u,0)+i(0, \lambda v)$$ for some 
$$(u,\lambda v) \in S^1 \times D^2_{[r/2,r]}$$ 
Then since  
$$\eta(u,\theta v) = (\cosh \theta u, \sinh \theta v)$$
and 
$$k_2((\cosh \theta u,0) + (0,\sinh \theta v)) = \sinh^2\theta + 2\sinh \theta$$
it follows that $\Phi^2_{-1/3}$ and $\eta$  differ only by a radial diffeomorphism
$$\theta \mapsto \lambda =  \sinh^2\theta + 2\sinh \theta.$$
(Above, $\varphi$ is the inverse of this map.)
\end{remark}

Here is the precise statement and proof of the claim in the last remark concerning 
symplectic transport on the real part.

\begin{lemma} \label{realtransport}
Let $\pi\co  E \into \bbC$ be a symplectic Lefschetz fibration, where $E$ is equipped with
symplectic structure $\omega$, such that the regular fibers of $\pi$ are symplectic, and 
$J$ is an almost complex structure on $E$ compatible with $\omega$ such that $\pi_*(Jv) = i J\pi_*(v)$. Suppose $E$ has an anti-symplectic, anti-complex involution
$$\iota\co  E \into E,$$
that is,
$$\iota^2 = id$$
$$\iota^*\omega = -\omega$$
$$\iota^*J = -J,$$
$$\pi(\iota(p)) = \overline{\pi(p)}.$$
Then the real part $Y = E^\iota$ is such that 
for each $x \in Y \setminus Crit(\pi)$
the symplectic lift $\xi \in T_x(E)$ of $\partial_t \in T_{\pi(x)}(\bbR)$
(i.e. $(D\pi)(\xi) = \partial_t$ and $\omega(\xi, v) =0$ for all $v  \in Ker(D\pi)$)
satisfies 
$$\xi \in T(Y),$$
and in fact 
$$ \xi = \nabla_{g}f/ | \nabla_{g}f|,$$
where 
$\pi = F+iH$, $f = \pi|_{Y} = F|_{Y}$, and $g = \omega_J$ (i.e. $g(v,w) = \omega(v,Jw)$).
This means the symplectic transport preserves $Y$, and coincides with the unit speed gradient flow of $f= F|_{Y}$. (Of course, symplectic transport only makes sense on $Y \setminus Crit(\pi)$.)
\end{lemma}

\begin{proof}  Let $p \in Y = E^\iota$. We will show that $\xi_p \in T(Y)$.
First note that $T_p(Y) = T_p(E)^{\iota_*}$. $\xi$ is characterized by:
$D\pi(\xi) =\partial_t$ and $\xi \in Ker( D\pi)^\omega$. Let us check that 
$\iota_*(\xi)$ also satisfies these. First, 
$$D \pi(\iota_*\xi) = (\iota_\bbC)_*(D\pi(\xi)) = (\iota_\bbC)_*(\partial_t) = \partial_t.$$
Next, let $v \in Ker(D\pi)$. Notice that $\iota_*(v) \in Ker(D\pi)$ also, by a similar calculation as above. Thus
$$\omega(\iota_* \xi, \iota_* v) = -\omega(\xi,v) =0$$
shows that $\iota_* \xi \in Ker(D\pi)^\omega$.
\\
\newline Now let's show that $\xi = \nabla f /|\nabla f|$, where $f = F|_{Y}$. 
First we show $\xi = \nabla F /|\nabla F|$.
Using $\pi_*(JV) = i\pi_*(v)$, it is easy to check that
$\nabla F =  \pm X_H$ and $\nabla H = \pm X_F$. Now observe that
$$D(\pi)^\omega = span\{ X_H, X_F\};$$
this follows at once from $\pi^{-1}(z) = F^{-1}(a) \cap (H)^{-1}(b)$,
where $z = a+bi$, and dimension considerations.
From this it follows immediately that $\nabla F /|\nabla F| \in D(\pi)^\omega$.
Furthermore 
$$D\pi(\nabla F /|\nabla F|) = DF(\nabla F /|\nabla F|) = 1 \cdot \partial_t$$
because $DH(\nabla F) = DH(\pm X_H) = 0$. Thus we have shown
$$\xi = \nabla F /|\nabla F|.$$
Now, since we already checked that $\xi \in T(Y)$, it follows that 
in fact $\xi = \nabla f /|\nabla f|$, where $f = F|_{Y}$. \end{proof}

As we explained in remark \ref{pivsfC} in the introduction, we expect to prove in a future paper that $E$ is conformally exact symplectomorphic to $D(T^*N)$. For now,  we give a detailed sketch of the proof of the following proposition, which states that $E$ is at least homotopy equivalent to $N$. 

\begin{prop} $E$ is homotopy equivalent to $N$.\label{Prophomotopyequiv}
\end{prop}
\begin{proof}[Sketch of proof] First it is well-known that $E$ is homotopy equivalent to the result of attaching one 
4-disk to $M$ at each vanishing sphere. (In fact, $E$ is diffeomorphic to 
the result of attaching one 4-handle to $M\times D^2$ at each vanishing sphere.
See \cite{GS} for the corresponding statement when $\dim E = 4$.) Let us denote the disks $\Delta_0, \Delta_2^{\, j}, \Delta_4$, where $\partial \Delta_0$ is attached to $L_0$ and so on.
\\
\newline Recall from \S \ref{4mfds} that there are exact Weinstein embeddings 
$$\phi_{L_2^{\, j}}\co  D_r(T^*S_j) \into M, \, S_j = S^3, \, j=1,\ldots, k$$
such that 
\begin{gather} \label{fr}
\phi_{L_2^{\, j}}|_{D_r(\nu^*K_-)} = \phi_j\co  S^1 \times D^2_r\into L_0,
\end{gather}
where we use the canonical identification $D_r(\nu^*K_-) \cong S^1 \times D^2_r$.
Let 
$$M_0 = D_R(T^*L_0) \cup (\cup_j \phi_{L_2^{\, j}}( D_r(T^*S_j)) \subset M.$$
Then $M_0$ is a retract of $M$. (In fact the homeomorphism $M \into M_0$ we mentioned in
\S \ref{plumbing} is a retract.)
\\
\newline 
Now, writing $D_r(T^*L_2^{\, j})$ for $\phi_{L_2^{\, j}}(D_r(T^*S_j))$, we have
$$E \backsimeq M_0 \cup (\Delta_0 \cup \Delta_2^{\, j} \cup \Delta_4)$$
$$= [D_R(T^*L_0)\cup ( \cup_j D_r(T^*L_2^{\, j}))] \cup (\Delta_0 \cup \Delta_2^{\, j} \cup \Delta_4)$$ 
$$\backsimeq  [D_R(T^*L_0)\cup \Delta_0] \cup [\cup_j (D_r(T^*L_2^{\, j}) \cup \Delta_2^{\, j} ) ]  \cup \Delta_4.$$
Note that $D_R(T^*L_0) \cup \Delta_0$ is homotopy equivalent to $\Delta_0$,
and we have 
$$K_j \subset \partial \Delta_0.$$
Now we define a certain  subset of $D_r(T^*L_2^{\, j})\cup \Delta_2^{\, j}$ which is diffeomorphic to a 2-handle $D^2 \times D^2$. Recall from section \ref{4mfds} that $L_4$ is the union of 
$$L_0 \setminus (\cup_j \phi_j(S^1 \times D^2_{r/2})) \text{ and } \cup_j L_4^{\, j},$$ 
where $L_4^{\, j} = \phi_{L_2^{\, j}}(\widetilde T)$. Now the definition of $\widetilde T$ (see \S \ref{surgeryh}) shows that $\widetilde T \setminus D_r(\nu^*K-)$ is the graph of a 1-form
$$\widetilde \alpha_4\co  S^3 \setminus K_- \into T^*(S^3 \setminus K_-).$$ 
Meanwhile 
$$\widetilde T \cap D_r(\nu^*K_-)  = D_{[r_0, r]}(\nu^*K_-)$$ 
for some $0< r_0 < r$.
Define $\widetilde H_2^{\, j} \subset D_r(T^*S_j)\cup \Delta_2^{\, j}$ by 
$$\widetilde H_2^{\, j} = \Delta_2^{\, j} \cup \{(p,v)\in T^*(S^3\setminus K_-) : p \in S^3 \setminus K_-, v = s \widetilde \alpha_4(p), \text{ for some } s \in [0,1] \} \cup D_{r_0}(\nu^*K_-^{\, j}).$$
Set
$$H_2^{\, j} = \phi_{L_2^{\, j}}(\widetilde H_2^{\, j}).$$ 
Then we claim that $H_2^{\, j}$ is homeomorphic to a 2-handle $D^2 \times D^2$,
where $\partial D^2 \times D^2$ corresponds to $D_{r_0}(\nu^*K_-^{\, j})$
and $D^2 \times \partial D^2$ corresponds to
$$\phi_{L_2^{\, j}}(\Gamma(\widetilde \alpha_4)) \cup S_{r_0}(\nu^*K_-^{\, j}).$$ 
(See figure \ref{homotopyequiv} for the picture corresponding to the case $\dim N = 2$, 
$\dim M = 2$, $\dim \Delta = 2$.) Furthermore we claim that $D_r(T^*L_2^{\, j}) \cup \Delta_2^{\, j}$ is homotopy equivalent to $H_2^{\, j}$ (by a retraction). (We omit the proofs of these claims.)

\begin{figure}
\begin{center}
\includegraphics[width=4in]{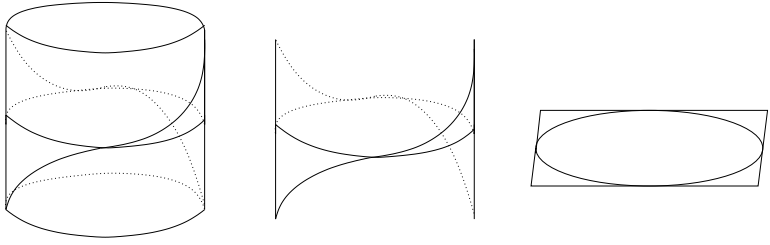} 
\caption{In the case $\dim N =2$, from left to right: 
$D_r(T^*S^1) \cup D^2 \backsimeq H_2 \cong \text{ a 2-handle}.$}
\label{homotopyequiv}
\end{center}
\end{figure}
Note the restriction 
$$\phi_{L_2^{\, j}}|_{ D_{r_0}(\nu^*K_-^{\, j})}\co  D_{r_0}(\nu^*K_-^{\, j}) \into L_0,$$
still makes sense on the retract $H_2^{\, j}$. And, as we noted before, this restriction is equal to 
the framing 
$$\phi_j\co  S^1 \times D^2_{r_0} \into L_0.$$
Now, using $[D_R(T^*L_0)\cup \Delta_0] \backsimeq \Delta_0$ and 
$(D_r(T^*L_2^{\, j}) \cup \Delta_2^{\, j} \backsimeq H_2^{\, j}$, we see that
$$[D_R(T^*L_0)\cup \Delta_0] \cup [\cup_j (D_r(T^*L_2^{\, j}) \cup \Delta_2^{\, j} )]$$
$$\backsimeq \Delta_0 \cup [\cup_j H_2^{\, j}].$$
Here, $\Delta_0 \cup [\cup_j H_2^{\, j}]$ is a partial handle decomposition (of $N$)
given by $D^4$ with $k$ 2-handles attached using the framings $\phi_j$.
Indeed, recall that  $H_2^{\, j} \cong D^2 \times D^2$ in such a way that 
$$\partial D^2 \times D^2 \cong D_{r_0}(\nu^*K_-^{\, j}),$$
and $\phi_{L_2^{\, j}}|_{D_{r_0}(\nu^*K_-^{\, j})} = \phi_j$.
Next, since $D^2 \times \partial D^2$ corresponds to
$$\Gamma(\alpha_2^{\, j}) \cup S_{r_0}(\nu^*K_-^{\, j}),$$
it follows that the boundary of $\Delta_0 \cup H_2^{\, j}$ is equal to $L_4$, which is the sphere where $\Delta_4$ is attached. To finish, we write 
$$E \backsimeq
[D_R(T^*L_0)\cup \Delta_0] \cup [\cup_j (D_r(T^*L_2^{\, j}) \cup \Delta_2^{\, j})]  \cup \Delta_4$$
$$\backsimeq \Delta_0 \cup [\cup_j H_2^{\, j}] \cup \Delta_4 \cong N.$$
For the last diffeomorphism, note that the left hand side  is a handle decomposition using the attaching maps $\phi_j$, hence it is a handle decomposition of $N$.
\end{proof}

\begin{remark}\label{refinements} Theorem \ref{main} can be refined in a couple of ways. First if $\xi$ is any gradient-like vector field for $f$ on $N$ we can construct a corresponding vectorfield  on $\widetilde N$, which is the symplectic lift of $\partial_t$ (up to a scaling function). To see the correspondence one compares the two vector fields on each handle, just like we did for $\pi|_{\widetilde N}$ and $f$.
Second, there is an anti-symplectic involution on $E$, say $\iota_{E}$, for which 
$\widetilde N$ is the fixed point set (as one would be the case if $E = T^*N$ and $\widetilde N = N$).
\\
\newline Here is a sketch of the construction of $\iota_{E}$. (Instead of anti-symplectic involution, we will say conjugation map.) On each local fibration $E_0^{loc}$, $E_2^{loc}$, $E_4^{loc}$
let $\iota_{E_i^{loc}}$ denote the standard conjugation. Also, let $\iota_{D^2}$ be the usual conjugation map on  $D^2 \subset \bbC$.  First we define a conjugation map $\iota_{E_0}$  on $E_0$ as follows. Think of $M$ as the plumbing (see \S \ref{plumbing}) of $D(T^*L_0)$ and $D(T^*L_2^{\,j})$, where $L_0 = L_2^{\,j} = S^3$. Then we define a conjugation map $\iota_M$  on $M$ as follows. The guiding idea is that $\iota_M$ is to have fixed point set equal to $L_0$. On $D(T^*L_0)$ it is defined to be $(x,y) \mapsto (x,-y)$; and on 
each $D(T^*L_2^{\,j})$ it is defined to be 
$$((x_1,x_2),(y_1,y_2)) \mapsto ((x_1,-x_2),(-y_1,y_2))$$
so that the fixed point set in $D(T^*L_2^{\,j})$ is $D(\nu^*K_-)$, which is identified with part of $L_0$. Then, when we glue $M \times D^2$ and $E_0^{loc}$ together to get $E_0$, the two conjugation maps
$\iota_M \times \iota_{D^2}$ and $\iota_{E_0^{loc}}$ will patch together to give a conjugation map on 
$\iota_{E_0}$ on $E_0$.
\\
\newline For $E_4$ and $E_2$ we define conjugation maps  $\iota_{E_2}$, $\iota_{E_4}$ in a similar way: On $M_2$, $\iota_{M_2}$ is defined the same way on $D(T^*L_0)$ but on $D(T^*L_2^{\, j})$ it is defined to be
$$((x_1,x_2),(y_1,y_2)) \mapsto ((-x_1,x_2),(y_1,-y_2))$$
so that the fixed point set is $D(\nu^*K_+)$, which is identified with part of $L_0$ to form $L_4'$.
Then  by combining $\iota_{M_2} \times \iota_{D^2}$ with  $\iota_{E_2^{loc}}$ and  $\iota_{E_4^{loc}}$ 
we get a conjugation maps $\iota_{E_2}$ on $E_2$ and  $\iota_{E_4}$ on $E_4$.
To combine $\iota_{E_0}$,$\iota_{E_2}$,$\iota_{E_4}$ to  get a conjugation map $\iota$ on $E$ one checks that the gluing maps $\pi_0^{-1}(1) \into \pi_2^{-1}(-1)$ and  $\pi_2^{-1}(1) \into \pi_4^{-1}(-1)$  map $\iota_{E_0}$ to $\iota_{E_0}$ and 
$\iota_{E_2}$ to $\iota_{E_4}$.
\end{remark}

\section{Construction of $\pi\co  E \into D^2$ and $N \subset E$, $\dim N = 3$}\label{3-manifolds}


Consider the case when $N$ is a closed 3-manifold, and $f\co  N \into \bbR$ is a 
self-indexing Morse function. The following  discussion applies equally well to self-indexing Morse functions $f\co  N \into \bbR$ with four critical values $0,n,n+1,2n+1$. 
See \S \ref{2nfiber} to see how things are much the same from one dimension to the next. In \S \ref{3mfds} we explained how to constructed a Weinstein manifold  $M$ with exact Lagrangian spheres $L_0, L_1^j, L_2^j, L_3$, one for each critical point of $f$.
Assume for simplicity of notation there is only one critical point of each index: 
$$x_3, x_2, x_1, x_0.$$
Thus $T = f^{-1}(3/2) \cong T^2$, a torus, and we have just one $\alpha$ curve and one $\beta$ curve.
For each $j=0,1,2,3$ we will define a Lefschetz fibration $\pi_j\co  E_j\into D(c_j)$ over a disk $D(c_j)$, where $c_j$ is the critical value of $\pi_j$, and $\pi_j$ has just one critical point
corresponding to $x_j$. Then we will fiber-connect sum the $\pi_j$'s together  to form  $\pi\co E\into S \cong D^2$.  Then we will show that $\pi$ has regular fiber isomorphic to $M$
and the vanishing spheres for suitable paths correspond to $L_0, L_2, L_2, L_3 \subset M$.
Finally, we will show there is an exact Lagrangian embedding $N \subset E$ 
such that $\pi|_{N} \cong f$.
\begin{remark} Of course, if  there were several critical points of, say  index 1, denoted $x_1^{\, j}$,
then $\pi_1$ would have one critical value $c_1$ and several critical points lying over $c_1$ corresponding to the $x_1^{\, j}$. The treatment is much the same in this case, as one can see from 
our treatment of $4-$manifolds earlier.
\end{remark}

Each $\pi_j$ will have a certain prescribed regular fiber $M_j$ (which will be a ``twist'' of $M$ depending on $j$, as in \S \ref{pi2}) and it will have one prescribed vanishing sphere (which will be a ``twist'' of $L_j \subset M$ depending on $j$). The base point will be $b = c_1+1/4$ and the vanishing paths will be as in figure \ref{figurevanishingpaths3mfds}.
\\
\newline  First,  set
$$M_1 = M$$
and construct  
$$\pi_1\co  E_1 \into D(c_1),$$
as in \S \ref{pi0} (but using the local model $\pi_{loc}^1$ based on $q_1=z_1^2 +z_2^2 -z_3^2$)
so that $\pi_1$ has fiber $\pi_1^{-1}(c_1+1)$ canonically isomorphic to  $M_1 = M$ with vanishing sphere corresponding to $L_1 \subset M$. 
\\
\newline Next, set 
$$M_2 = T^{L_2}_{\pi/2}(M)$$ 
where $T^{L_2}_{\pi/2}$ is analogous to the twist operation we defined in \S \ref{pi2}. 
It is easy to see that $M_2 = T^{L_2}_{\pi/2}(M)$ has an exact Lagrangian 
sphere say $L_2' \subset M_2$ which corresponds in a straight-forward way to $L_2 \subset M_1 =M$.
(The precise definition of $L_2$ is analogous to that of $(L_2^{\, j})' \subset M_2$ for any $j$, in \ref{pi2}.)
Define $\pi_2\co  E_2 \into D(c_2)$ as in \S \ref{pi0} (but using the local model $\pi_{loc}^2$ based on $q_2=z_1^2 -z_2^2 -z_3^2$) so that $\pi_2$ has fiber $\pi_2^{-1}(c_2+1)$ canonically isomorphic to $M_2$ and vanishing sphere corresponding to $L_2' \subset M_2$. 
\\
\newline Now define 
$$M_0 = T^{L_1}_{-\pi/2}(M).$$ 
By a construction analogous to that in \S \ref{pi4}, there is an exact  Lagrangian sphere $L_0'$ corresponding to $L_0 \subset M$. It is an ``untwisted'' version of $L_0$. (The precise definition is analogous to that of  $L_4' \subset M_2$ in \S \ref{pi4}.) Roughly speaking,
$L_0'$ is the union of $T \setminus N(\alpha)$ and $\phi_{L_1}(D_r(\nu^*K_-))$, where $N(\alpha)$ is a tubular neighborhood of $\alpha$, and $D_r(\nu^*K_-)$ is ``flat'', i.e. it is no longer ``twisted''. 
Construct $\pi_0\co  E_0 \into D(c_0)$ as in \S \ref{pi0} (but using the local model $\pi_{loc}^0$ based on $q_0= z_1^2 +z_2^2 +z_3^2$) so that $\pi_0$ is a Lefschetz fibration with fiber 
$\pi_0^{-1}(c_0+1)$ canonically isomorphic to $M_0$ and vanishing sphere corresponding to $L_0' \subset M_2$.
\\
\newline Finally let 
$$M_3 = M_2.$$
There is an exact Lagrangian submanifold $L_3' \subset M_3= M_2$ which corresponds to 
$L_3 \subset M$. Again, it is the ''untwisted'' version of $L_3 \subset M$, analogous to
$L_0' \subset M_0$ before.
Construct  $\pi_3\co  E_3 \into D(c_3)$  as in \S \ref{pi4} (using the local model 
$\pi_{loc}^3$ based on $q_3=-z_1^2 -z_2^2 -z_3^2$) so that $\pi_3$ is a Lefschetz fibration with fiber 
$\pi_3^{-1}(c_3-1)$ canonically isomorphic to $M_3$.
(Note we have used $c_3-1$ here rather than $c_3+1$, as in \S \ref{pi4}.)
\\
\newline In summary, there are canonical isomorphisms
$$\pi_0^{-1}(c_0+1) \cong T_{-\pi/2}^{L_1}(M), \, \pi_1^{-1}(c_1+1) \cong M,$$
$$\, \pi_2^{-1}(c_2+1) \cong T_{\pi/2}^{L_2}(M), \, \pi_3^{-1}(c_3-1)\cong T_{\pi/2}^{L_2}(M),$$
and the vanishing spheres for $\pi_0, \ldots, \pi_3$ correspond to 
$$L_0' \subset  T_{-\pi/2}^{L_1}(M), \, L_1 \subset M, \, L_2' \subset  T_{\pi/2}^{L_2}(M), \text{ and } \, L_3' \subset  T_{\pi/2}^{L_2}(M).$$

Now, by an analogue of lemma \ref{leftfiber=M}, there are canonical isomorphisms
$$\pi_1^{-1}(c_1-1) \cong   T_{-\pi/2}^{L_1}(M), \text{ and } \pi_2^{-1}(c_2+1) \cong T_{-\pi/2}^{L_2}(T_{\pi/2}^{L_2}(M)) \cong M.$$
Thus we have identifications 
$$\pi_0^{-1}(c_0 +1) \cong \pi_1^{-1}(c_1-1), \pi_1^{-1}(c_1+1) \cong \pi_2^{-1}(c_2-1),\pi_2^{-1}(c_2 +1) \cong \pi_3^{-1}(c_3-1)$$
and we use these to fiber connect sum  $\pi_0, \pi_1, \pi_2, \pi_3$ (see \S \ref{pi}) to get a Lefschetz fibration
$$\pi\co  E \into S$$
where $S \cong D^2$. 
\\
\newline To find the vanishing spheres of $\pi$, choose the base point at $c_1 +1/4$ and choose transport maps analogous to those in 
\S \ref{computingVC}. That is they are straight lines in $\bbR$, except when they approach a critical value, in which case they do a half arc in the lower half plane to avoid the critical value.  See figure \ref{figurevanishingpaths3mfds}.
\begin{figure}
\begin{center}
\includegraphics[width=4in]{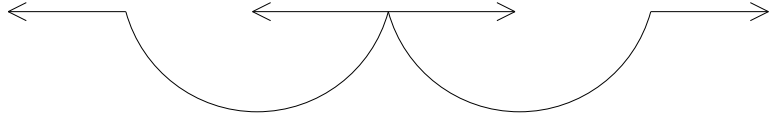}
\caption{The four vanishing paths for $(E,\pi)$ in the case when 
$\dim N = 3$.}
\label{figurevanishingpaths3mfds}
\end{center}
\end{figure}

Then, by an analogue of lemma \ref{halftwist}, there are symplectomorphisms
$$\tau_1\co  M \into T_{-\pi/2}^{L_1}(M), \text{ and } \tau_2\co  M \into T_{\pi/2}^{L_2}(M)$$
such that 
$$\tau_1(L_0) = L_0', \, \tau_2(L_2) = L_2', \, \text{ and }\tau_2(L_3) = L_3'.$$
Moreover, under the above identifications, $\tau_1$ and $\tau_2$ correspond to the transport maps along the half arcs in the lower half plane, for $0< s \leq 1$:
$$\pi_1^{-1}(c_1 +s) \into \pi_1^{-1}(c_1 -s), \text{ and } \pi_1^{-1}(c_1 +s) \cong  \pi_2^{-1}(c_2 -s) \into \pi_2^{-1}(c_1 +s).$$ 
(Here, the isomorphism $\pi_1^{-1}(c_1 +s) \cong  \pi_2^{-1}(c_2 -s)$ is the transport map along 
the segment $[c_1+s, c_2-s]$; it does not have much geometric effect.)
By an argument similar to that in the proof of lemma \ref{vanishingcycles}, it is easy to see that the vanishing spheres in  
$\pi_1^{-1}(c_1 +s) \cong M$  correspond exactly to $L_0, L_1, L_2, L_3 \subset M$. 
\\
\newline To construct $N \subset E$, we take $N_0 \subset E_0$ and $N_3\subset E_3$ 
to be Lefschetz thimbles over $[c_0, c_0+s]$ and $[c_3-s, c_3]$ for some small $s>0$, as before.
These correspond to the 0- and 3-handles of  a Milnor-type handle decomposition for $N$.
Next, we define a subset  $N_1^{loc} \subset E_{loc}^1 \subset E_1$.
Let
$$f_1= q_1|_{E_{loc}^1 \cap \bbR^3} \co  E_{loc}^1 \cap \bbR^3 \into \bbR,$$
and let 
$$N_1^{loc} = E_{loc}^1 \cap \bbR^3  \cap f_1^{-1}([c_1-1/2, c_1+1/2]).$$
This corresponds to the $1-$handle of $N$.  
Now recall that $E_1$ is a certain quotient  of the disjoint union 
$$[M_1 \setminus (\phi_{L_1}(D_{r/2}(T^*S^2)))] \times D^2(c_1) \sqcup E_{loc}^1.$$
Recall $\phi_{L_1}$ is such that $\phi_{L_1}(D_{r/2}(\nu^*K_-))$ is a tubular neighborhood of
$\alpha$ in $T =T^2$, namely $\phi^{\alpha}(S^1 \times D^1_{1/2})$.
Define \begin{eqnarray*}
N_1^{triv} = [T \setminus \phi_{L_1}(D_{r/2}(\nu^*K_-))] \times [c_2-1/2, c_2 +1/2] \\
\subset [M_1 \setminus \phi_{L_1}(D_{r/2}(T^*S^2)))] \times D^2_{r/2}(c_2)
\end{eqnarray*}
Then, $N_1 \subset E_1$ is defined to be the union
$$N_1 = N_1^{loc} \cup N_1^{triv}.$$
We claim that $N_1$ is diffeomorphic to $f^{-1}([1-s, 1+s])$ for some small $s>0$, so that it is a cobordism between $S^2$ and $T^2$.
Indeed, following the proof of Theorem \ref{main} we see that in the quotient space $E_1$, 
$$N_1^{loc}\cap (k^1)^{-1}((4(r/2)^2,4r^2]) \subset N_1^{loc}$$ is identified with 
$$\phi^{\alpha}(S^1 \times D_{(r/2, r]}^2) \times [c_2-1/2, c_2+1/2] \subset N_1^{triv}.$$
Therefore the union $N_1$ is smooth, and because we are using the correct framing $\phi^\alpha$ (in Milnor's handle-type decomposition \cite[pages 27-32]{Milnor}), $N_1$ is  diffeomorphic to $f^{-1}([1-s, 1+s])$ for some small $s>0$.
\\
\newline Moreover, under the canonical isomorphism $\pi_1^{-1}(c_1-1) \cong   T_{-\pi/2}^{L_1}(M)$, the boundary component of $N_1$ which lies in $\pi_1^{-1}(c_1-s)$ corresponds 
precisely to $L_0'$. Indeed, that boundary component of $N_1$ is by definition 
equal to the union of  
$$[T \setminus \phi^{\alpha}(S^1 \times D^1_{r/2})] \times \{c_1 -1/2\}$$ 
and
$$N_1^{loc}\cap f^{-1}(c_1 -1/2) \cong S^0 \times D^2.$$ 
Under the isomorphism
$\pi_1^{-1}(c_1-1) \cong   T_{-\pi/2}^{L_1}(M)$, this union corresponds to 
the union of $[T \setminus \phi^{\alpha}(S^1 \times D^1_{r/2})] \times \{c_1 -1/2\}$ and
$D_r(\nu^*K_+) \subset D_r(T^*L_1)$. This is precisely $L_0'$, by definition.
This implies that $N_1$ and $N_0$ glue up smoothly (with over-lapping boundary)
to yield a manifold diffeomorphic to $f^{-1}([0,1+s])$ for some small $s>0$.
\\
\newline We define $N_2^{loc} \subset E_{loc}^1 \subset E_1$ and
$f_2\co  N_2^{loc} \into [c_2-1/2, c_2+1/2]$ in a similar way.
Then 
$$N_2^{triv} \subset [M_2 \setminus \phi_{L_2'}(D_{r/2}(T^*S^2))] \times [c_2-1/2, c_2+1/2]$$  
is defined to be 
$$N_2^{triv} = L_3' \setminus \phi_{L_2'}(D_{r/2}(\nu^*K_+))\times [c_2-1/2, c_2+1/2].$$
As before $N_2^{triv}$ and  $N_2^{loc}$ glue together to form a  manifold diffeomorphic to
$f^{-1}([c_2-s, c_2+s])$. The boundary component of 
$$N_2 = N_2^{loc} \cup N_2^{triv}$$
which lies in $\pi_2^{-1}(c_2 +1/2)$ corresponds to $L_3'$. (This is immediate in this case.) This means $N_2$ and
$N_3$ glue up smoothly to form a manifold diffeomorphic to $f^{-1}([2-s,3])$, for some small $s>0$.
And, as before, the boundary component of 
$$N_2 = N_2^{loc} \cup N_2^{triv}$$
which lies in $\pi_2^{-1}(c_2 -1/2)$ corresponds to $T$ under the canonical isomorphism
$\pi_2^{-1}(c_2 -1/2) \cong M$. Therefore $N_0 \cup N_1$ and $N_2 \cup N_3$ glue up smoothly along
$T$ to form a manifold 
$$\widetilde N =  N_0 \cup N_1 \cup N_2 \cup N_3$$ 
diffeomorphic to $N$. To see that $\pi$ is equivalent to $f$ we just compare $\pi|_{N_j}$ to
$f$ on the corresponding handle in the Milnor type decomposition of $N$.

%
%
%
%

\end{document}